\documentclass{amsart}

\usepackage{graphicx}
\usepackage{amsmath}
\usepackage{mathrsfs}
\usepackage{amsmath, amscd}
\usepackage{amssymb}
\usepackage[hyphens]{url}
\usepackage{tikz}
\usepackage{tikz-cd}
\usepackage{enumitem}
\newtheorem{theorem}{Theorem}[section]
\newtheorem*{theorem*}{Theorem}
\newtheorem{lemma}[theorem]{Lemma}
\newtheorem{proposition}[theorem]{Proposition}
\newtheorem{corollary}[theorem]{Corollary}
\theoremstyle{definition}

\theoremstyle{remark}
\newtheorem{remark}[theorem]{Remark}
\newtheorem{remarkthm}{Remark}[theorem]
\numberwithin{equation}{section}

\newcommand{\abs}[1]{\lvert#1\rvert}
\newcommand{\Z}{\mathbb{Z}}
\newcommand{\opn}{\operatorname}
\newcommand{\wt}{\widetilde}

\newcommand{\scP}{\mathscr{P}}
\newcommand{\scS}{\mathscr{S}}
\newcommand{\testleft}{\leftarrow\!\shortmid}

\begin{document}

\title[Some Torsion Classes in $A^*_{PGL_n}$ and $H^*_{PGL_n}$]{Some Torsion Classes in the Chow Ring and Cohomology of $BPGL_n$}

\author{Xing Gu}

\address{Max Planck Institute for Mathematics, Vivatsgasse 7, 53111 Bonn, Germany.}
\email{gux2006@outlook.com}
\thanks{Part of this paper was written when the author was a research fellow at the University of Melbourne, Australia, supported by the Australian research council, and the rest of this paper was written during a visit to the Max Planck Institute for Mathematics, Germany. The author is grateful to the University of Melbourne, the Australian research council and the Max Planck Institute for Mathematics for their supports.}





\begin{abstract}
 In the integral cohomology ring of the classifying space of the projective linear group $PGL_n$ (over $\mathbb{C}$),  we find a collection of $p$-torsion classes $y_{p,k}$ of degree $2(p^{k+1}+1)$ for any odd prime divisor $p$ of $n$, and $k\geq 0$.

 If, in addition, $p^2\nmid n$, there are $p$-torsion classes $\rho_{p,k}$ of degree $p^{k+1}+1$ in the Chow ring of the classifying stack of $PGL_n$, such that the cycle class map takes $\rho_{p,k}$ to $y_{p,k}$.

 We present an application of the above classes regarding Chern subrings.
 \end{abstract}

\maketitle
\tableofcontents

\section{Introduction}\label{Introduction}
Let $G$ be an algebraic (resp. topological) group and $\mathbf{B}G$ be the classifying stack (resp. space) of $G$. We follow Vistoli's notations \cite{Vi} and let $A^*_{G}$ (resp. $H^*_{G}$) denote the Chow ring (resp. singular cohomology ring with coefficient ring $\Z$.) of $\mathbf{B}G$.

The Chow ring of $\mathbf{B}G$ is first introduced by Totaro \cite{Tot}, as an algebraic analog of the integral cohomology of the classifying space of a topological group. Based on this work, Edidin and Graham \cite{Ed} developed equivariant intersection theory, of which the main object of interest is the equivariant Chow ring, an algebraic analog of Borel's equivariant cohomology theory.

When the algebraic group $G$ is over the base field of complex numbers $\mathbb{C}$, it has an underlying topological group which we also denote by $G$, and there is a cycle class map
\[\opn{cl}: A^*_G\rightarrow H^{*}_G\]
from $A^i_G$ to $H^{2i}_G$, which plays a crucial role in this paper. The cycle class map factors though the even degree subgroup of $MU^*(\mathbf{B}G)\otimes_{MU^*}\mathbb{Z}$
\begin{equation}\label{refinedcl}
    \opn{cl}: A^*_G\xrightarrow{\tilde{\opn{cl}}} MU^*(\mathbf{B}G)\otimes_{MU^*}\mathbb{Z}\rightarrow H^{*}_G
\end{equation}
where $\tilde{\opn{cl}}$ is called the \textit{refined cycle class map}.

The Chow rings $A^*_G$ have been computed by Totaro \cite{Tot} for $G=GL_n$, $SL_n$, $Sp_n$. (Hereafter, the base field is always $\mathbb{C}$ except for $O_n$, $SO_n$ and $Spin_n$.) For $G=SO_{2n}$ by Field \cite{Fi}, for $G=O_n$ and $G=SO_{2n+1}$ by Totaro \cite{Tot} and Pandharipande \cite{Pa}, for $G=\opn{Spin}_7$ by Guillot \cite{Gui2},
for $G=\opn{Spin}_8$ by Rojas \cite{Ro}, for the semisimple simply connected group of type $G_2$ by Yagita \cite{Ya}. The case that $G$ is finite has been considered by Guillot \cite{Gui} and \cite{Gui1}. In \cite{Ro1}, Rojas and Vistoli provided a unified approach to the known computations of $A^*_G$ for the classical groups $G=GL_n$, $SL_n$, $Sp_n$, $O_n$, $SO_n$.


Let $PGL_n$ be the quotient group of $GL_n$ over $\mathbb{C}$, modulo scalars. The case $G=PGL_n$ appears considerably more difficult than the cases mentioned above. The mod $p$ cohomology for some special choices of $n$ is considered by Toda \cite{To}, Kono and Mimura \cite{Ko}, Vavpeti{\u c} and Viruel \cite{Va}. In \cite{Gu}, the author computed the integral cohomology of $PGL_n$ for all $n>1$ in degree less then or equal to $10$ and found a family of distinguished elements $y_{p,0}$ of dimension $2(p+1)$ for any prime divisor $p$ of $n$ and any $k\geq 0$. (In the papers above, what are actually considered are the projective unitary groups $PU_n$, of which the classifying spaces are homotopy equivalent to those of $PGL_n$.) In \cite{Ve}, Vezzosi almost fully determined the ring structure of $A^*_{PGL_3}$. In \cite{Vi}, Vistoli determined the graded abelian group structures of $A^*_{PGL_p}$ and $H^*_{PGL_p}$ for $p$ an odd prime, and partially determined their ring structures. In Kameko and Yagita \cite{Ka}, the additive structure of $A^*_{PGL_p}$ is obtained independently, as a corollary of the main results.

Classes in the rings $A^*_{PGL_n}$ (resp. $H^*_{PGL_n}$) are important invariants for sheaves of Azumaya algebras of degree $n$ (resp. principal $PGL_n$-bundles.) For instance, the ring $H^*_{PGL_n}$ plays a key role in the topological period-index problem, considered by Antieau and Williams \cite{An}, \cite{An1},  Crowley and Grant \cite{Cr}, and the author in \cite{Gu1}, \cite{Gu2}. This suggests that the $A^*_{PGL_n}$ could be of use in the study of the (algebraic) period-index problem due to J.-L. Colliot-Th{\'e}l{\`e}ne \cite{Co}. The cohomology ring  $H^*_{PGL_n}$ also appears in the study of anomalies in particle physics such as \cite{cordova2020anomalies}, \cite{davighi2019global} and
\cite{Ga}. Antieau \cite{antieau2015integral} and  Kameko \cite{Ka2} considered the cohomology of $\mathbf{B}G$ for some finite cover $G$ of $PGL_{p^2}$ to construct counterexamples for the integral Tate conjecture.

In this paper we study the torsion elements in $A^*_{PGL_n}$ and $H^*_{PGL_n}$. One may easily find $H^i_{PGL_n}=0$ for $i=1,2$ and $H^3_{PGL_n}\cong\mathbb{Z}/n$. Therefore we have a map
\begin{equation}\label{eq:chi}
\chi:\mathbf{B}PGL_n\rightarrow K(\mathbb{Z},3)
\end{equation}
representing the ``canonical'' (in the sense to be explained in Section \ref{SecSpecSeq}) generator $x_1$ of $H^3_{PGL_n}$, where $K(\mathbb{Z},3)$ denotes the Eilenberg-Mac Lane space with the $3$rd homotopy group $\mathbb{Z}$.

In principle, the cohomology of $K(\mathbb{Z},3)$ and more generally of $K(\pi, n)$ with $n>0$ and $\pi$ a finitely generated abelian group are determined by Cartan and Serre \cite{Ca}. Based on their work, the author gave an explicit description of the cohomology of $K(\mathbb{Z},3)$ in \cite{Gu}, which is outlined as follows. Let $p$ be a prime number. The $p$-local cohomology ring of $K(\mathbb{Z},3)$ is generated by the fundamental class $x_1$ and $p$-torsion classes of the form $y_{p,I}$, where $I=(i_m,\cdots,i_1)$ is an ordered sequence of integers $0\leq i_m<\cdots<i_1$. Here $m$ is called the \textit{length} of $I$ and denoted by $l(I)$. For $I=(k)$, we simply write $y_{p,k}$ for $y_{p,I}$. The degree of $y_{p,I}$ is
\begin{equation}\label{eq:deg}
\opn{deg}(y_{p,I})=1+\sum_{j=1}^m(2p^{i_j+1}+1).
\end{equation}
In \cite{Gu} the author showed that the images of $y_{p,0}$ in $H^{2(p+1)}_{PGL_n}$ via $\chi$ are nontrivial $p$-torsion classes, for all prime divisors $p$ of $n$. In this paper we generalize this to $y_{p,k}$ for $k>0$ and under some restrictions, to Chow rings. For simplicity we omit the notation $\chi^*$ and let $y_{p,I}$ denote $\chi^*(y_{p,I})$ whenever there is no risk of ambiguity.

\begin{theorem}\label{thm2p^k+2}
\begin{enumerate}
\item In $H^{2(p^{k+1}+1)}_{PGL_n}$, we have $p$-torsion classes $y_{p,k}\neq 0$ for all odd prime divisors $p$ of $n$ and $k\geq 0$.

\item If, in addition, we have $p^2\nmid n$, then there are $p$-torsion classes $\rho_{p,k}\in A^{p^{k+1}+1}_{PGL_n}$ mapping to $y_{p,k}$ via the cycle class map.
\end{enumerate}

\end{theorem}

\begin{remarkthm}
The sets of elements $\{\rho_{p,k}\}_{k\geq 0}$ and $\{y_{p,k}\}_{k\geq 0}$ are not algebraically independent in general. Indeed, in the case $n=p$ considered by Vistoli \cite{Vi}, they each generate subalgebras isomorphic to a subalgebra of the polynomial algebra (over $\mathbb{Z}/p$) of $2$ generators. (See Corollary \ref{transdimension}).
\end{remarkthm}

One the other hand, the classes $y_{p,I}$ for $l(I)>1$ seem hard to capture, and so are their counterparts in Chow rings (if any). Nonetheless, we are able to find an interesting property of these classes. By definition, we have $y_{p,I}\in\opn{Im}\chi^*$, a subring of $H^*_{PGL_n}$. In a non-negatively graded ring $R^*$ with $R^0=\Z$, an element $x$ is \emph{decomposable} if $x=\sum_i y_iz_i$ where $y_i$, $z_i$ are of degree greater then $0$.

\begin{theorem}\label{thmypI}
Let $p$ be a (not necessarily odd) prime divisor of $n$ and $p^2\nmid n$. If $l(I)>1$, then the class $y_{p,I}$ is decomposable in $\opn{Im}\chi^*$.
\end{theorem}
\begin{remarkthm}
In Corollary \ref{cor:ypI} we show that for some $I$ with $l(I)>1$, the class $y_{p,I}\in\opn{Im}\chi^*$ is nonzero. This suggests that the homomorphism $\chi^*$ sends a nontrivial class of the form
\[y_{p,I}+(\textrm{decomposable classes})\in H^*(K(\Z,3);\Z_{(p)})\]
to $0$.
\end{remarkthm}
We proceed to discuss an application of Theorem \ref{thm2p^k+2}, for which some background is in order. Let $G$ be a compact topological (resp. algebraic) group. Let $h^*(-)$ be a generalized cohomology such that
\[h^*(\mathbf{B}GL_r)\cong h^*(pt)[c_1, \cdots, c_r]\]
where $c_i\in h^{2i}(\mathbf{B}GL_r)$ is the $i$th universal Chern class. (resp. Let $h^*(-)$ be $A^*(-)$, or $A^*(-)\otimes\mathbb{Z}_{(p)}$, i.e., Chow ring localized at an odd prime $p$.) The \textit{Chern subring} of $h^*_G:=h^*(\mathbf{B}G)$ is the subring generated by Chern classes of all representations of $G\rightarrow GL_r$ for some $r$. If the Chern subring is equal to $h^*_G$, then we say that $h^*_G$ is \textit{generated by Chern classes}.

Whether $A^*_G$ is generated by Chern classes is related to a conjecture by Totaro \cite{Tot} on the image of the refined cycle class map, to be discussed in Section \ref{SecChern}. Vezzosi \cite{Ve} showed that $A^*_{PGL_3}$ is not generated by Chern classes. In the appendix of \cite{Vi} by Targa \cite{Ta}, the same thing was shown for $A^*_{PGL_p}$ for all odd primes $p$. More results of this nature can be found in Kameko and Yagita \cite{Ka1}. We have the following consequence and generalization of the above results:
\begin{theorem}\label{thmChern}
Let $n>1$ be an integer, and $p$ one of its odd prime divisor, such that $p^2\nmid n$. Then the ring $A^*_{PGL_n}\otimes\mathbb{Z}_{(p)}$ is not generated by Chern classes. More precisely, the class $\rho_{p,0}^i$ is not in the Chern subring for $p-1\nmid i$.
\end{theorem}

This paper is organized as follows. In Section \ref{SecKZ3} we recall the cohomology of the Eilenberg-Mac Lane space $K(\mathbb{Z},3)$ in terms of the Steenrod reduced power operations. In Section \ref{SecSpecSeq} we recall a homotopy fiber sequence and the associated Serre spectral sequence converging to $H^*_{PGL_n}$, which is the main object of interest in \cite{Gu}. We prove Theorem \ref{thmypI} in this section. Section \ref{SecEquiv} is a recollection of elements in equivariant intersection theory related to the topic of this paper. Section \ref{SecMot} is a short note on the Steenrod reduced power operations for motivic cohomology, which play a role in the proof of Theorem \ref{thm2p^k+2}. Here we concern ourselves with some specific facts rather than the general theory. In Section \ref{SecVistoli} we review and slightly extend works of Vezzosi \cite{Vi} and Vistoli \cite{Ve} on the cohomology and Chow ring of $\mathbf{B}PGL_p$ where $p$ is an odd prime. We prove part (1) of Theorem \ref{thm2p^k+2} in this section. In Section \ref{SecBlock}, we consider the subgroups of certain diagonal block matrices of $GL_n$ pass to quotients to yield subgroups of $PGL_n$ that act as a bridge between $PGL_p$ and $PGL_n$.
In Section \ref{SecPer} and \ref{SecRho}, we construct the torsion classes $\rho_{p,k}$, completing the proof of Theorem \ref{thm2p^k+2}. Finally, in Section \ref{SecChern} we discuss the conjecture by Totaro on the refined cycle class map and its relation to the Chern subrings, followed by the proof of Theorem \ref{thmChern}.

\textbf{Acknowledgement. }The author thanks Professor Christian Haesemeyer for financial support as well as many fruitful conversations. Without his wisdom and encouragement the author would have gone through much trouble completing this manuscript. The author also thanks the anonymous referees for numerous suggestions and corrections.
\section{The cohomology of the Eilenberg-Mac Lane space $K(\mathbb{Z},3)$}\label{SecKZ3}
The cohomology of the Eilenberg-Mac Lane space $K(\pi,n)$ for $\pi$ a finitely generated abelian group can be deduced from the lecture notes \cite{Ca} by Cartan and Serre. The mod $p$ cohomology of $K(\pi,n)$ is described particularly nicely by Tamanoi \cite{tamanoi1999subalgebras} in terms of the Milnor basis of the mod $p$ Steenrod algebra. As a special case of the above, the cohomology of $K(\Z,3)$ is described by the author in \cite{Gu} in full details. The author claims no originality of the content of the above or this section. For the sake of simplicity, we present the $p$-local case for odd primes $p$. Let $\mathbb{Z}_{(p)}$ be the ring of integers localized at $p$.

\begin{proposition}[\cite{Gu}, Proposition 2.16, 2016]\label{K(Z,3)general}
The graded ring $H^*(K(\mathbb{Z},3);\mathbb{Z}_{(p)})$ is generated by
\begin{enumerate}
\item $x_1\in H^3(K(\mathbb{Z},3);\mathbb{Z}_{(p)})$, a non-torsion element, and
\item
the elements $y_{p,I}$ where $I=(i_m,\cdots,i_1)$ is an ordered sequence of positive integers $i_m<\cdots <i_1$, and the degree of $y_{p,I}$ is
\[\abs{y_{p,I}}=1+\sum_{j=1}^m(2p^{i_j+1}+1).\]
In particular, when taking $I=(i)$, we have $y_{p,i}:=y_{p,I}$ of degree $2p^{i+1}+2$.
\end{enumerate}
\end{proposition}
Here $x_1$ is the fundamental class, i.e, the class represented by the identity map of $K(\mathbb{Z},3)$. Notice that there are nontrivial relations among the elements $y_{p,I}$. See \cite{Gu} for details. For future references we consider the decomposable elements of degree $\abs{I}$.
\begin{lemma}\label{lem: degree of I}
Let $I=(i_m,\cdots,i_1)$ and $J=(j_n,\cdots,j_1)$ satisfying $0\leq i_m<\cdots<i_1$ and $0\leq j_n<\cdots<j_1$. Then we have $I=J$ if and only if $\abs{I}=\abs{J}$.
\end{lemma}
\begin{proof}
Only one direction requires a proof. For $m=1$, $\abs{I}=2p^{i_1+1}+2$ is the $p$-adic expansion of $\abs{I}$, for both $p$ odd and $2$. The lemma then follows from the uniqueness of $p$-adic expansions. The general case then follows by induction on $m$.
\end{proof}

\begin{corollary}\label{cor:decomp}
Let $I=(i_m,\cdots,i_1)$ satisfying $0\leq i_m<\cdots<i_1$, and $Y\in H^{\abs{I}}(K(\Z,3);\Z_{(p)})$ be a monomial in elements $y_{p,J}$ for various $J$. Then $Y$ is either equal to $y_{p,I}$ or decomposable.
\end{corollary}

Proposition \ref{K(Z,3)general} implies, in particular, that any torsion element in $H^*(K(\mathbb{Z},3);\mathbb{Z})$ is of order $p$ for some prime number $p$. This means that the torsion elements of $H^*(K(\mathbb{Z},3);\mathbb{Z})$, namely all elements above degree $3$, are in the image of some Bockstein homomorphism
\[B:H^*(-;\mathbb{Z}/p)\rightarrow H^{*+1}(-;\mathbb{Z}).\]
So we present the classes $y_{p,k}$ in such a way.

For a fixed odd prime $p$, let $\mathscr{P}^k$ be the $k$th Steenrod  reduced power operation and let $\beta$ be the Bockstein homomorphism $H^*(-;\mathbb{Z}/p)\rightarrow H^*(-;\mathbb{Z}/p)$ associated to the short exact sequence
\[0\rightarrow\mathbb{Z}/p\rightarrow\mathbb{Z}/p^2\rightarrow\mathbb{Z}/p\rightarrow 0.\]
For an axiomatic description of the Steenrod reduced power operations, see Steenrod and Epstein \cite{St}.

To compute the cohomology ring $H^*(K(\mathbb{Z},3);\mathbb{Z}/p)$, it suffices to consider the mod $p$ cohomological Serre spectral sequence associated to the homotpy fiber sequence
\[K(\mathbb{Z},2)\rightarrow PK(\mathbb{Z},3)\rightarrow K(\mathbb{Z},3)\]
where $PK(\mathbb{Z},3)$ is the contractible space of paths in $K(\mathbb{Z},3)$.

Now let overhead bars denote the mod $p$ reductions of integral cohomology classes. An inductive argument on the transgressive elements (Section 6.2 of  McCleary \cite{Mc}) yields the following
\begin{proposition}\label{SteenrodOperation0}
For any odd prime $p$, we have
\[H^*(K(\mathbb{Z},3);\mathbb{Z}/p)\cong\Lambda_{\mathbb{Z}/p}[\bar{x}_1]\otimes
\Lambda_{\mathbb{Z}/p}[x_{p,0},x_{p,1},\cdots]\otimes\mathbb{Z}/p[\bar{y}_{p,0},\bar{y}_{p,1},\cdots],\]
where $\Lambda_{\mathbb{Z}/p}[\bar{x}_1]$ is the exterior algebra generated by $\bar{x}_1$, the mod $p$ reduction of the fundamental class $x_1$, $\Lambda_{\mathbb{Z}/p}[x_{p,0},x_{p,1},\cdots]$ is the exterior algebra over $\mathbb{Z}/p$ over elements
\[x_{p,k}:=\mathscr{P}^{p^k}\mathscr{P}^{p^{k-1}}\cdots\mathscr{P}^{1}(\bar{x}_1), \textrm{ where }k=0,1,2,\cdots,\]
and $\mathbb{Z}/p[\bar{y}_{p,0},\bar{y}_{p,1},\cdots,]$ is the polynomial algebra generated by
\[\bar{y}_{p,i}=\beta x_{p,i}.\]

In $H^*(K(\mathbb{Z},3);\mathbb{Z})$, we have
\[y_{p,I}=B(x_{p,i_m}\cdots x_{p,i_2}x_{p,i_1})\]
for $I=(i_m,\cdots,i_2,i_1)$.
\end{proposition}

An immediate consequence of Proposition \ref{SteenrodOperation0} is the following
\begin{proposition}\label{SteenrodOperation1}
For $k\geq0$, we have
\[y_{p,k}=B\mathscr{P}^{p^k}\mathscr{P}^{p^{k-1}}\cdots\mathscr{P}^{1}(\bar{x}_1).\]
\end{proposition}
\begin{remark}
The alert reader may argue that Proposition \ref{SteenrodOperation0} merely indicates
\[y_{p,k}=\lambda B\mathscr{P}^{p^k}\mathscr{P}^{p^{k-1}}\cdots\mathscr{P}^{1}(\bar{x}_1)\]
for some $\lambda\in(\mathbb{Z}/p)^{\times}$. However, notice that Proposition \ref{K(Z,3)general} determines $y_{p,k}$ only up to a scalar multiplication. Hence we may as well choose $y_{p,k}$ such that Proposition \ref{SteenrodOperation1} holds.
\end{remark}

Proposition \ref{SteenrodOperation1} has the following variation.
\begin{proposition}\label{SteenrodOperation2}
For $k\geq1$, we have
\[\bar{y}_{p,k}=\mathscr{P}^{p^k}(\bar{y}_{p,k-1}).\]
\end{proposition}
\begin{proof}
Recall that for positive integers $a,b$ such that $a\leq pb$, we have the Adem relation (Adem, \cite{Ad1})
\begin{equation}\label{Adem}
\begin{split}
\mathscr{P}^a\beta\mathscr{P}^b=\sum_i (-1)^{a+i}{(p-1)(b-i)\choose a-pi}\beta\mathscr{P}^{a+b-i}\mathscr{P}^i\\+\sum_{i} (-1)^{a+i+1}{(p-1)(b-i)-1\choose a-pi-1}\mathscr{P}^{a+b-i}\beta\mathscr{P}^i.
\end{split}
\end{equation}

For $k>0$, let $a=p^k$ and $b=p^{k-1}$. Then the only choice of $i$ to offer something nontrivial on the right hand side of (\ref{Adem}) is $i=p^{k-1}$, and (\ref{Adem}) becomes
\begin{equation}\label{Adem1}
\mathscr{P}^{p^k}\beta\mathscr{P}^{p^{k-1}}=\beta\mathscr{P}^{p^k}\mathscr{P}^{p^{k-1}}.
\end{equation}
Then it follows by induction that we have
\[\mathscr{P}^{p^k}\beta\mathscr{P}^{p^{k-1}}\cdots\mathscr{P}^p\mathscr{P}^1=
\beta\mathscr{P}^{p^k}\cdots\mathscr{P}^p\mathscr{P}^1.\]
Since all classes $y_{p,k}$ are of order $p$, as stated in Proposition \ref{SteenrodOperation1}, we conclude.
\end{proof}
\begin{remark}\label{rem:Milnor basis}
Alternatively, we may describe $\bar{y}_{p,k}$ in terms of the Milnor's basis $\{Q_k\}_{k\geq0}$ introduced in Section 6 of \cite{milnor1958steenrod}, defined inductively by
\begin{equation*}
Q_0=\beta,\hspace{2 mm} Q_{k+1}=\mathscr{P}^{p^k}Q_k-Q_k\mathscr{P}^{p^k}.
\end{equation*}
Then by Proposition \ref{SteenrodOperation1} we may follow an inductive argument and show
\[\bar{y}_{p,k}=Q_{k+1}(\bar{x}_1.)\]
\end{remark}
\section{A Serre spectral sequence, the proof of Theorem \ref{thmypI}}\label{SecSpecSeq}
In the introduction we mentioned the map
\[\chi: \mathbf{B}PGL_n\rightarrow K(\mathbb{Z},3)\]
representing the ``canonical'' class in $H^3_{PGL_n}$. It is easy to show, for example, in the introduction of Gu \cite{Gu}, that the homotopy fiber is $\mathbf{B}GL_n$ and there is a homotopy fiber sequence
\begin{equation}\label{fiberseq}
\mathbf{B}GL_n\rightarrow\mathbf{B}PGL_n\xrightarrow{\chi}K(\mathbb{Z},3)
\end{equation}
where the first arrow is induced by the obvious projection  $GL_n\to PGL_n$. This may be obtained from the more obvious homotopy fiber sequence
\begin{equation}\label{fiberseq1}
\mathbf{B}\mathbb{C}^{\times}\rightarrow\mathbf{B}GL_n\rightarrow\mathbf{B}PGL_n
\end{equation}
by de-looping the first term $\mathbf{B}\mathbb{C}^{\times}\simeq K(\mathbb{Z},2)$. The author used this de-looping as the definition of $\chi$, and the class $x_1$ is ``canonical'' in this sense.

For positive integers $r>1$ and $s$ we have $\Delta: \mathbf{B}PGL_{r}\rightarrow\mathbf{B}PGL_{rs}$ the obvious diagonal map. The fiber sequences (\ref{fiberseq}) and (\ref{fiberseq1}) then show the following
\begin{lemma}\label{diag}
The following diagram commutes up to homotopy.
\begin{equation*}
\begin{tikzcd}
\mathbf{B}PGL_p\arrow[rr, "\Delta"]\arrow[dr, "\chi"]&&\mathbf{B}PGL_n\arrow[dl, "\chi"]\\
&K(\mathbb{Z},3)
\end{tikzcd}
\end{equation*}
\end{lemma}
\begin{remark}\label{abuse y}
As mentioned in the introduction, we omit the notation $\chi^*$ and let $y_{p,I}$ denotes their pull-backs in $H^*_{PGL_n}$. If we temporarily denote $y_{p,I}$ by $y_{p,I}(r)$ and $y_{p,I}(rs)$, respectively in the cases $n=r$ and $n=rs$, then Lemma \ref{diag} indicates $\Delta^*(y_{p,I}(rs))=y_{p,I}(r)$. In view of this, we do not make the effort of distinguishing $y_{p,I}(r)$ and $y_{p,I}(rs)$, and denote both by $y_{p,I}$.
\end{remark}

Consider the homotopy commutative diagram
\begin{equation}\label{3by3diagram}
\begin{CD}
  \mathbf{B}\mathbb{C}^{\times} @>\simeq >> K(\mathbb{Z},2)        @>>>   PK(\mathbb{Z},3)         @>>>  K(\mathbb{Z},3)\\
  \Phi:         @.        @VV \mathbf{B}\varphi V      @VVV                        @|\\
                @.        \mathbf{B}T(GL_n)        @>>> \mathbf{B}T(PGL_n)         @>>>   K(\mathbb{Z},3)\\
  \Psi:         @.          @VV\mathbf{B}\psi V                         @VV\mathbf{B}\psi' V    @|\\
                @.        \mathbf{B}GL_{n}         @>>> \mathbf{B}PGL_{n}    @>>> K(\mathbb{Z},3)\\
  \end{CD}
\end{equation}
where $PK(\mathbb{Z},3)$ is the pointed path space of $K(\mathbb{Z},3)$, which is contractible, and $\varphi$ is the diagonal map. Moreover, all rows are homotopy fiber sequences and we let $\Phi$ and $\Psi$ in the diagram denote morphisms of homotopy fiber sequence.
\begin{remark}
As mentioned in the introduction, the diagram (\ref{3by3diagram}) is indeed a variation of the one considered in \cite{Gu}, where the groups $U_n$, $PU_n$ and the unit circle take the places of $GL_n$, $PGL_n$ and $\mathbb{C}^{\times}$, respectively.
\end{remark}

We take note of the following notations:
\begin{equation}\label{S1}
H^*_{\mathbb{C}^{\times}}=\mathbb{Z}[v]
\end{equation}
where $v$ is of degree $2$, and
\begin{equation}\label{T(GLn)}
H^*_{T(GL_n)}=\mathbb{Z}[v_1,v_2,\cdots,v_n]
\end{equation}
where each $v_i$ is of degree $2$, and for all $i$ we have
\[\mathbf{B}\varphi^*(v_i)=v.\]
The quotient map $T(GL_n)\rightarrow T(PGL_n)$ identifies $H^*_{PGL_n}$ as the subring of $H^*_{T(GL_n)}$, which, as an abelian group, is generated by $1$ and the kernel of the ring homomorphism given by
\[H^*_{T(GL_n)}=\mathbb{Z}[v_1,v_2,\cdots,v_n]\rightarrow\Z,\hspace{2 mm}v_i\mapsto 1,\]
i.e., polynomials in $v_1,v_2,\cdots,v_n$ of which the coefficient of all term sum up to $0$.
Moreover, we have
\begin{equation}\label{GLn}
H^*_{GL_n}=\mathbb{Z}[c_1,c_2,\cdots,c_n]
\end{equation}
where $c_i$, the $i$th universal Chern class, is of degree $2i$, and $\mathbf{B}\psi^*$ takes $c_i$ to the $i$th elementary symmetric polynomial in $v_i$'s.

As in \cite{Gu}, we let $({^KE}_*^{*,*},{^Kd}_*^{*,*})$, $({^T}E_*^{*,*},{^Td}_*^{*,*})$ and $({^UE}_*^{*,*},{^Ud}_*^{*,*})$ denote cohomological Serre spectral sequences with integer coefficients associated to the three homotopy fiber sequences in (\ref{3by3diagram}). In particular, we have
\[^UE_2^{s,t}\cong H^s(K(\mathbb{Z},3);H^t_{GL_n})\]
converging to $H^*_{PGL_n}$. This spectral sequence is the main object of interest in \cite{Gu}.
In principle, using the homological algebra of differential graded algebras, we are able to determine all the differentials of ${^KE}_*^{*,*}$. The diagram (\ref{3by3diagram}) then converts the differentials ${^Kd}_*^{*,*}$ to ${^Td}_*^{*,*}$. Then we obtain some of the differentials ${^Ud}_*^{*,*}$ via the following
\begin{lemma}\label{lem:Udiff}
The morphism of spectral sequence ${^UE}_*^{*,*}\rightarrow {^TE}_*{*,*}$ is induced by
\begin{equation*}
\begin{split}
{^UE}_2^{*,*}\cong H^*(K(\Z,3);\Z)\otimes \Z[c_1,\cdots,c_n] &\rightarrow  H^*(K(\Z,3);\Z)\otimes\Z[t_1,\cdots,t_n]\cong{^TE}_2^{*,*}\\
hc_i &\mapsto h\sigma_i(t_1,\cdots,t_n),
\end{split}
\end{equation*}
where $h\in H^*(K(\Z,3);\Z)$ and $\sigma_i(t_1,\cdots,t_n)$ is the $i$th elementary symmetric polynomial in $t_1,\cdots,t_n$.
\end{lemma}
\begin{remark}
For the sake of simplicity we drop the symbol $\otimes$ whenever there is no ambiguity.
\end{remark}
We recall the splitting principle, which asserts that the map
\[\mathbf{B}\psi:\mathbf{B}T(GL_n)\rightarrow \mathbf{B}GL_n\]
associated to the inclusion of a maximal torus, induces the homomorphism
\begin{equation}\label{eq:Bpsi}
\mathbf{B}\psi^*: H^*_{GL_n}\rightarrow H^*_{T(GL_n)},\hspace{3 mm}c_i\mapsto\sigma_i(t_1,\cdots,t_n),
\end{equation}
where $\sigma_i(t_1,\cdots,t_n)$ is the $i$th elementary symmetric polynomial in variables $v_1,\cdots,v_n$. The following proposition then follows.

\begin{proposition}[Proposition 3.8, \cite{Gu}]\label{pro:Udiff}
The diagram \eqref{3by3diagram} induces a commutative diagram as follows:
\begin{equation*}
\begin{tikzcd}
{^UE}_{r}^{s,t}\arrow[d,"\Psi^*"]\arrow[r,"{^Ud}_{r}^{s,t}"]&
{^UE}_{r}^{s+r,t-r+1}\arrow[d,"\Psi^*"]\\
{^TE}_{r}^{s,t}\arrow[d,"\Phi^*"]\arrow[r,"{^Td}_{r}^{s,t}"]&
{^TE}_{r}^{s+r,t-r+1}
\arrow[d,"\Phi^*"]\\
{^KE}_{r}^{s,t}\arrow[r,"{^Kd}_{r}^{s,t}"]&{^KE}_{r}^{s+r,t-r+1}
\end{tikzcd}
\end{equation*}
where the arrow $\Psi^*$ is given by \eqref{eq:Bpsi}, in the sense that the source (resp. target) of $\Psi^*$ is a subquotient of the source (resp. target) of the homomorphism
\[{^UE}_{2}^{s,t}\cong H^s(BGL_n;H^t(K(\Z,3);\Z))\rightarrow H^s(BT(GL_n;)H^t(K(\Z,3);\Z))\cong {^TE}_{2}^{s,t}\]
induced by \eqref{eq:Bpsi}. In particular, the differentials $^Td_*^{*,*}$ determine all of the differentials of the form ${^Ud}_{r}^{s-r,t+r-1}$ of ${^UE}_{*}^{*,*}$ such that for any $r'<r$, ${^Td}_{r'}^{s-r',t-r'+1}=0$, since in this case the arrow $\Psi^*$ on the right is injective.
\end{proposition}
In other words, for each bi-degree $(s,t)$, the first nontrivial differential whose target is $^UE_*^{s,t}$ is determined by restricting $^Td_*^{*,*}$ to $^UE_*^{*,*}$. In particular, Proposition \ref{pro:Udiff} determines the differential $^Ud_3^{*,*}$ completely:
Let
\[\nabla:H^*(\mathbf{B}U_n)\rightarrow H^{*-2}(\mathbf{B}U_n)\]
be a linear operator defined by
\begin{equation}\label{eq:nabla ck}
\nabla(c_k)=(n+k-1)x_1c_{k-1}
\end{equation}
and the Leibniz formula
\begin{equation}\label{eq:nabla Leibniz}
\nabla(ab)=\nabla(a)b+a\nabla(b).
\end{equation}
We end this section by several corollaries. Let $\nabla$, $\xi$ be the same as earlier, and let $\vartheta=\vartheta(c_1, \cdots, c_n)\in H^{*}(BU_{n};\mathbb{Z})$. Then we have

\begin{corollary}[Corollary 3.4 and Corollary 3.10, \cite{Gu}]\label{cor:Udiff}
The differential ${^Ud}_3$ is determined by
 \[^{U}d_{3}(\xi\vartheta)={^{U}d}_{3}(\xi\vartheta)=\xi x_1\nabla(\vartheta)\]
for any $\xi\in H^*(K(\Z,3);\Z)$. In particular, we have
\[^{U}d_{3}(\xi c_1)={^{U}d}_{3}(c_1)\xi=nx_{1}\xi.\]
\end{corollary}

It is known (Theorem 1.2, \cite{Gu}) that for a prime $p$, the class $y_{p,0}\in H^{2(p+1)}_{PGL_n}$ is a nontrivial $p$-torsion class if $p|n$ and is $0$ otherwise. Since $y_{p,0}$ is the $p$-torsion class in $H^*(K(\mathbb{Z},3);\mathbb{Z})$ of the lowest degree $2(p+1)$, and generates the unique $p$-primary subgroup of degree $2(p+1)$. We have the following
\begin{proposition}\label{p-torsion lowest}
If $p|n$, then the class $y_{p,0}$ is the nontrivial $p$-torsion class in $H^*_{PGL_n}$ of lowest degree, and generates the unique $p$-primary subgroup of $H^{2(p+1)}_{PGL_n}$. If $p\nmid n$ the class $y_{p,0}$ is trivial in $H^*_{PGL_n}$.
\end{proposition}
\begin{proof}
The class $y_{p,0}$ being (non)trivial is just Theorem 1.2 of \cite{Gu}. The uniqueness assertion follows by looking at
\[^UE_2^{s,t}\cong H^s(K(\mathbb{Z},3);H^t_{GL_n}).\]
\end{proof}
Let $\Z_{(p)}$ denote the ring $\Z$ localized at the prime number $p$, and for any topological group $G$ let
\[H^*_{G}[p]:=H^*(\mathbf{B}G;\Z_{(p)})\cong H^*(\mathbf{B}G;\Z)\otimes\Z_{(p)}.\]
Recall that the ``canonical'' maximal torus , $T(PGL_n)$, is the subgroup of $PGL_n$ of diagonal matrices passing to the quotient. The Weyl group of $T(PGL_n)$ is the symmetric group $S_n$ acting by permuting the diagonal entries. It is a standard fact that the restriction $H^*_{PGL_n}[p]\rightarrow H^*_{T(PGL_n)}[p]$ factors through $(H^*_{T(PGL_n)}[p])^{S_n}$, the subring of $H^*_{T(PGL_n)}[p]$ of invariants with respect to the $S_n$-action. Of course this is also true without localization.
\begin{corollary}\label{cor:H^2p+2 direct sum}
The inclusion of the maximal torus $T(PGL_n)\hookrightarrow PGL_n$ and the map $\chi: \mathbf{B}PGL_n\rightarrow K(\Z,3)$ induce a split short exact sequence
\[0\rightarrow H^{2(p+1)}(K(\Z,3);\Z_{(p)})\rightarrow H^{2(p+1)}_{PGL_n}[p]\rightarrow (H^{2(p+1)}_{T(PGL_n)}[p])^{S_n}\rightarrow0,\]
which yields an isomorphism
\begin{equation}\label{eq:H^2p+2 direct sum}
H^{2(p+1)}_{PGL_n}[p]\cong (H^{2(p+1)}_{T(PGL_n)}[p])^{S_n}\oplus H^{2(p+1)}(K(\Z,3);\Z_{(p)}).
\end{equation}
\end{corollary}
\begin{proof}
Let ${^UE}_*^{*,*}[p]$ denote the spectral sequence ${^UE}_*^{*,*}$ localized at $p$. Then the nontrivial entries of its second page of total degree $2(p+1)$ are ${^UE}_2^{0,2(p+1)}[p]$ and, when $p|n$, ${^UE}_2^{2(p+1),0}[p]$. Then we have a short exact sequence
\[0\rightarrow {^UE}_{\infty}^{2(p+1),0}[p] \rightarrow H^{2(p+1)}_{PGL_n}[p]\rightarrow {^UE}_{\infty}^{0,2(p+1)}[p]\rightarrow0\]
which is split since ${^UE}_{\infty}^{0,2(p+1)}[p]$ is a torsion-free $\Z_{(p)}$-module.

The identification of the first term of the short exact sequence follows immediately from Proposition \ref{p-torsion lowest}. A direct computation identifies $(H^*_{T(PGL_n)})^{S_n}$ with $\opn{Ker}\nabla\cong {^UE}_3^{0,*}$. However, for obvious degree reasons there is no nontrivial differential out of $^{UE}_3{0,*}[p]$, and we have the identification of the last term of the short exact sequence.
\end{proof}
We proceed to consider the classes $y_{p,I}$ for $I$ of length greater than $1$, and prove Theorem \ref{thmypI}. It follows from Corollary 2.18 of \cite{Gu} that for each $I=(i_m,i_{m-1},\cdots,i_1)$ and an integer $k$ such that $0\leq i_m<\cdots<i_1$ we have
\begin{equation}\label{d(ypI)}
y_{p,I}=
\begin{cases}
{^Kd}_{2p^{i_m+1}+1}(y_{p,I'}v^{p^{i_m+1}}),\hspace{1 mm}m>1,\\
{^Kd}_{2p^{i_1+1}-1}(x_1v^{p^{i_1+1}}).
\end{cases}
\end{equation}
with $I'=(i_{m-1},\cdots,i_1)$.
\begin{lemma}\label{lem:diffdecomp}
For $I=(i_m,\cdots,i_1)$ as above with $m>1$. Then for $r<2p^{i_m+1}+1$, any element in the image of
\begin{equation}\label{eq:lower dr}
^Kd_r^{\abs{I}-r,r-1}:{^KE}_r^{\abs{I}-r,r-1}\rightarrow {^KE}_r^{\abs{I},0}
\end{equation}
is congruent to a decomposable element, in the sense that ${^KE}_r^{\abs{I},0}$ is a quotient group of $H^{\abs{I}}(K(\Z,3);\Z)$.
\end{lemma}
\begin{proof}
The differentials $^Kd_r^{s,t}$ landing on the line $^KE_r^{*,0}$ are determined by \eqref{d(ypI)} along with the product formula
\[^Kd_r(ab)={^Kd}_r(a)b+(-1)^{\abs{a}}a\hspace{1 mm}{^Kd}_r(b)\]
where $a$ and $b$ are classes in $^KE_r^{*,*}$ and $\abs{a}$ is the degree of $a$. Therefore, any differential as in  \eqref{eq:lower dr} with $r<2p^{i_m+1}+1$ is of the form
\[{^Kd}_{2p^{k+1}+1}^{\abs{I}-2p^{k+1}-1,2p^{k+1}}:{^KE}_{2p^{k+1}+1}^{\abs{I}-2p^{k+1}-1,2p^{k+1}}\rightarrow {^KE}_{2p^{k+1}+1}^{\abs{I},0}\]
with $k<i_m$. The image of ${^Kd}_{2p^{k+1}+1}^{\abs{I}-2p^{k+1}-1,2p^{k+1}}$ is therefore generated by monomials of the form $y_{p,I_1}\cdots y_{p,I_s}$ where at least one of the $I_1,\cdots, I_s$ is of the form $(k,\cdots)$. It then follows from Lemma \ref{lem: degree of I} that the monomial $y_{p,I_1}\cdots y_{p,I_s}$ is not $y_{p,I}$. By Corollary \ref{cor:decomp}, the monomial $y_{p,I_1}\cdots y_{p,I_s}$ is decomposable and we conclude.

\end{proof}
Now we have all the necessary ingredient for the following
\begin{proof}[Proof of Theorem \ref{thmypI}]

Let $I=(i_m,\cdots,i_1)$ such that $0\leq i_m<i_{m-1}<\cdots<i_1$. Let $W_U$, $W_T$ and $W_K$ be the subgroups of $H^{\abs{I}}(K(\Z,3);\Z)$ characterised by
\begin{equation*}
\begin{split}
&{^UE}_{2p^{i_m}}^{\abs{I},0}=H^{2p^{i_m}}_{GL_{n}}/W_U,\hspace{2 mm}
{^TE}_{2p^{i_m}}^{\abs{I},0}=H^{2p^{i_m}}_{T(GL_{n})}/W_T,\\
&{^KE}_{2p^{i_m}}^{\abs{I},0}=H^{2p^{i_m}}(K(\Z,3);\Z)/W_K.
\end{split}
\end{equation*}
Then we have $W_U\subset W_T\subset W_K$.
Summarizing the above arguments, we have the following commutative diagram:
\begin{equation*}
\begin{tikzcd}
{^UE}_{2p^{i_m}}^{\abs{J},2p^{i_m}}\arrow[d,"\Psi^*"]\arrow[r,"{^Ud}_{2p^{i_m}}^{\abs{I},2p^{i_m}}"]&
{^UE}_{2p^{i_m}}^{\abs{I},0}\arrow[d,"\Psi^*"]\cong H^{\abs{I}}(K(\Z,3);\Z)/W_U\\
{^TE}_{2p^{i_m}}^{\abs{J},2p^{i_m}}\arrow[d,"\Phi^*"]\arrow[r,"{^Td}_{2p^{i_m}}^{\abs{I},2p^{i_m}}"]&
{^TE}_{2p^{i_m}}^{\abs{I},0}
\arrow[d,"\Phi^*" near start, near end]\cong H^{\abs{I}}(K(\Z,3);\Z)/W_T\\
{^KE}_{2p^{i_m}}^{\abs{J},2p^{i_m}}\arrow[r,"{^Kd}_{2p^{i_m}}^{\abs{I},2p^{i_m}}"]&{^KE}_{2p^{i_m}}^{\abs{I},0}
\cong H^{\abs{I}}(K(\Z,3);\Z)/W_K
\end{tikzcd}
\end{equation*}
where $\Psi^*$ and $\Phi^*$ are induced from $\Psi$ and $\Phi$ in the diagram \eqref{3by3diagram}. The vertical arrows to the right are quotient maps.
Now it follows from \eqref{d(ypI)} that we have
\begin{equation*}
\begin{split}
&\Phi^*\Psi^*{^Ud}_{2p^{i_m}+1}(c_p^{p^{i_m}}y_{p,J})\\
=&{^Kd}_{2p^{i_m}+1}([y_{p,J}\cdot ({n\choose p}v)^{i_m+1}])\\
=&[{n\choose p}^{i_m+1}y_{p,I}]\in H^{\abs{I}}(K(\Z,3);\Z)/W_K
\end{split}
\end{equation*}
where $[a]$ denotes the equivalence class of $a\in H^*(K(\Z,3);\Z)$ in $H^*(K(\Z,3);\Z)/W_K$, and we have
\[{^Ud}_{2p^{i_m}+1}(c_p^{p^{i_m}}y_{p,J})\equiv {n\choose p}^{i_m+1}y_{p,I}\pmod{W_K}.\]
Since $p^2\nmid n$, we have ${n\choose p}\not\equiv 0\pmod{p}$, the above indicates that in ${^UE}_{\infty}^{2p^{i_m},0}$ we have
\[[y_{p,I}]\equiv 0\pmod {{^UE}_{\infty}^{2p^{i_m},0}\bigcap W_K}.\]
On the other hand, it follows from Lemma \ref{lem:diffdecomp} the classes in $W_K$ are decomposable, and we conclude.
\end{proof}

\section{Equivariant intersection theory}\label{SecEquiv}
We refer to Edidin and Graham \cite{Ed} and Totaro \cite{Tot} for definitions and basic facts on equivariant intersection theory. We also reformulate many results in Section 4 of Vistoli's paper \cite{Vi}, since they play key roles in this paper.

The main object of interest is the equivariant Chow ring $A^*_G(X)$ for an algebraic space $X$ over a base field $\mathbb{K}$ with an action of an algebraic group $G$. The Chow ring of $G$ is identified with the equivariant Chow ring $A^*_G(\opn{spec}\mathbb{K})$. In this sense the ring $A^*_G$ is regarded as the ``coefficient ring'' of the equivariant Chow rings $A^*_G(X)$. One may regard equivariant Chow rings as an analog of Borel's equivariant cohomology theory for topological spaces with continuous group actions. From now on we work with a fixed base field $\mathbb{K}$ and the reader is free to assume $\mathbb{K}=\mathbb{C}$.



For any $k>0$, Edidin and Graham define the Chow ring $A_G^*(X)$ for a scheme $X$ by defining $A^{\leq k}_G(X)$ as
$A^{\leq k}((X\times U)/G)$ where $A^*(-)$ denotes the ordinary Chow ring and $U$ is an open set of a $G$-representation of large enough dimension and $G$ acts freely on $U$. It follows that we have

\begin{proposition}(Homotopy Invariance)\label{Homotopy Invariance}
Let $X$ be a $G$-equivariant algebraic space and $V$ be a finite dimensional $G$-representation. Then the pullback of the projection
\[A^*_G(X)\rightarrow A^*_G(X\times V)\]
is an isomorphism. In particular, we have $A^*_G\cong A^*_G(V)$.
\end{proposition}

In general the quotient $X/G$ is defined as an algebraic space, which is not necessarily a scheme. However, by Lemma 9 of \cite{Ed}, for any $i>0$, there is a $G$-representation $V$ and an open subscheme $U$ of $V$ such that $V-U$ is of codimension greater than $i$, and the quotient $U\rightarrow U/G$ exists as a principal bundle in the category of schemes. Therefore, for any $i>0$, the Chow groups $A^i_G$ are defined as Chow groups of some schemes.

Many properties of the ordinary Chow rings hold for equivariant Chow rings too. For instance, let $f:Y\rightarrow X$ be a proper morphism of $G$-schemes, then we have the pullback
\[f^*:A^*_G(X)\rightarrow A^*_G(Y)\]
and the push-forward
\[f_*:A^{*-r}_G(Y)\rightarrow A^*_G(X),\]
where $r$ is the codimension of $Y$ in $X$, and the following proposition generalizes the localization sequences in the non-equivariant intersection theory. (See, for example, Fulton \cite{Fu}.)

\begin{proposition}(Localization Sequences)\label{Localization Sequence}
Let $f:Y\rightarrow X$ be $G$-equivariant closed immersion. Then we have an exact sequence as follows:
\[A^{*-r}_G(Y)\xrightarrow{f_*}A^*_G(X)\xrightarrow{f^*}A^*_G(X\backslash Y)\rightarrow 0.\]
\end{proposition}
In the particular case where $X=V$ is a $G$-representation of dimension $n$, and $Y=\{0\}\subset V$, we have
\begin{proposition}\label{Chern class}
The sequence
\[A^{*-n}_G\xrightarrow{c_n(V)}A^*_G\rightarrow A^*_G(V\backslash\{0\})\rightarrow 0\]
is exact, where the first arrow is the multiplication by the Chern class $c_n(V)$.
In particular, it follows from Proposition \ref{Homotopy Invariance} that we have
\[A^*_G(V\backslash\{0\})\cong A^*_G/c_n(V).\]
\end{proposition}
Now we let the group $G$ vary. Suppose $H$ is a closed subgroup of $G$. Then we have the restriction map
\[\opn{res}^G_H: A^*_G(X)\rightarrow A^*_H(X)\]
which is a ring homomorphism. If, furthermore, $H$ has finite index in $G$, then we have the transfer map
\[\opn{tr}^H_G: A^*_H\rightarrow A^*_G.\]
This is no longer a ring homomorphism, but a homomorphism of $A^*_G$-modules, in the sense that we have the following projection formula:
\begin{equation}\label{projformula}
\opn{tr}^H_G(a\opn{res}^G_H(b))=\opn{tr}^H_G(a)b.
\end{equation}
For the unit $1\in A^*_G$ we have
\begin{equation*}
\opn{tr}^H_G(1)=[G:H],
\end{equation*}
where the righthand side is the index of $H$ in $G$. Therefore, we have
\begin{equation}\label{projformula'}
\opn{tr}^H_G(\opn{res}^G_H(b))=[G:H]b.
\end{equation}

Similar to the Cartan-Eilenberg double coset formula (Adem and Milgram \cite{Ad}), we have Mackey's formula concerning the transfer and the restriction described above. Once again we adopt the setup in Vistoli \cite{Vi}:

Let $G$ be an algebraic group and $H$, $K$ be algebraic subgroups of $G$ such that $H$ has finite index over $G$. We will also assume that the quotient $G/H$ is reduced, and a disjoint union of copies of $\opn{spec}\mathbb{K}$ (this is automatically verified when $\mathbb{K}$ is algebraically closed of characteristic $0$). Furthermore, we assume that every element of $(K\backslash G/H)(\mathbb{K})$ is the image of
some element of $G(\mathbb{K})$.

Let $\mathscr{C}$ be a set of representatives of classes of the double quotient $K\backslash G/H(\mathbb{K})$. For each $s\in\mathscr{C}$, let
\[K_s:=K\bigcap sHs^{-1}\subset G.\]
Therefore, $K_s$ is a subgroup of $K$ of finite index, and there is an embedding $K_s\rightarrow H$ defined by $k\mapsto sks^{-1}$.
\begin{proposition}[Vistoli, Proposition 4.4, \cite{Vi}, 2007](Mackey's formula)\label{Mackey's formula}
\[\opn{res}^G_K\cdot\opn{tr}^H_G=\sum_{s\in\mathscr{C}}\opn{tr}^{K_s}_K\cdot\opn{res}^{sHs^{-1}}_{K_s}\cdot\gamma_s: A^*_H\rightarrow A^*_K,\]
where $\gamma_s$ is the restriction associated to the conjugation $sHs^{-1}\rightarrow H$.
\end{proposition}
There is another way to relate equivariant Chow rings over different algebraic groups. Suppose we have a monomorphism of algebraic groups $H\rightarrow G$. Let $H$ act on a scheme $X$. Then we have the $G$ equivaiant algebraic space $G\times^H X$, which is the quotient $G\times X/\sim$ where the equivalence relation ``$\sim$'' is defined by $(g,hx)\sim (gh,x)$ for all $h\in H$.
\begin{proposition}[Vistoli, \cite{Vi}, 2007]\label{induced space}
The composite of the restriction $A^*_G(G\times^H X)\rightarrow A^*_H(G\times^H X)$ and the pullback $A^*_H(G\times^H X)\rightarrow A^*_H(X)$ is an isomorphism.
\end{proposition}



Let $N_G$ be the normalizer of the maximal torus $T(G)$. It is well known (Gottlieb \cite{Go}) that the restriction homomorphism $H^*_G\rightarrow H^*_{N_G}$ is injective. The analog conclusion for Chow rings, according to Vezzosi \cite{Ve}, is shown in an unpublished work by Totaro. A sketch of the proof is presented in \cite{Ve}.

\begin{theorem}[Gottlieb, Totaro, Theorem 2.1, \cite{Ve}]\label{TotaroInjection}
Let $G$ be an algebraic group over $\mathbb{C}$, $T$ a maximal torus of $G$ and $N_G$ its normalizer in $G$. The restriction maps
\[A^*_G\rightarrow A^*_{N_G}\]
and
\[H^*_G\rightarrow H^*_{N_G}\]
are injective.
\end{theorem}

In general, Chow rings are much more complicated then singular cohomology. However, in many cases the Chow rings $A^*_G$ behave in very similar ways to $H^*_G$, their topological counterparts. We end this section with such examples, which will be of use later on.
\begin{proposition}[Totaro, \cite{Tot}, 1999]\label{pro:SLn}
For $G=GL_n$, $SL_n$ or a torus, the cycle class map $\opn{cl}: A^*_G\rightarrow H^*_{G}$ is an isomorphism of rings.
\end{proposition}

Furthermore, for $T(G)$ a fixed maximal torus of $G$, the cycle class map $\opn{cl}$ preserves the actions of the Weyl group. In particular, we have
\begin{corollary}\label{cor:cl T(PGLn)}
Let $G$ be an algebraic group, $T(G)$ a maximal torus, and $W$ its Weyl group. Then
\[\opn{cl}: (A^*_{T(G)})^{W}\rightarrow (H^*_{T(G)})^{W}\]
is a well-defined ring isomorphism.
\end{corollary}
On the other hand, we have the following
\begin{theorem}[Totaro, Theorem 2.14, \cite{totaro2014group}]\label{thm:Totaro}
For any affine algebraic group $G$ over $\mathbb{C}$, the natural map
\[A^*_G\otimes\mathbb{Q}\rightarrow H^*_G\otimes\mathbb{Q}\]
is an isomorphism.
\end{theorem}
The following is an immediate consequence of Corollary \ref{cor:cl T(PGLn)} and Theorem \ref{thm:Totaro}.
\begin{proposition}\label{Maximal torus}
Let $G$ be an affine algebraic group $G$ over $\mathbb{C}$. The restriction homomorphism induced from $T(G)\rightarrow G$ gives an isomorphism
\[A^*_G\otimes\mathbb{Q}\rightarrow (A^*_{T(G)})^W\otimes\mathbb{Q}.\]
\end{proposition}
\begin{proof}
Consider the following commutative diagram
\begin{equation*}
\begin{tikzcd}
A^*_G\otimes\mathbb{Q}\arrow[r]\arrow[d,"\opn{cl}"]&(A^*_{T(G)})^W\otimes\mathbb{Q}\arrow[d,"\opn{cl}"]\\
H^*_G\otimes\mathbb{Q}\arrow[r]&(H^*_{T(G)})^W\otimes\mathbb{Q}.
\end{tikzcd}
\end{equation*}
It follows from Corollary \ref{cor:cl T(PGLn)} and Theorem \ref{thm:Totaro} that the vertical arrows are isomorphisms. The bottom horizontal arrow being an isomorphism is a well-known fact, for which one may refer to, for example, Chapter 3 of \cite{hsiang2012cohomology}. It follows that the top horizontal arrow is an isomorphism as well.
\end{proof}
In particular, we have
\begin{corollary}\label{cor:T(PGL_n)Q}
The inclusion of a maximal torus $\lambda:T(PGL_n)\rightarrow PGL_n$ induces the following isomorphism:
\[\lambda^*:A^*_{PGL_n}\otimes\mathbb{Q}\xrightarrow{\cong} (A^*_{T(PGL_n)})^{S_n}\otimes\mathbb{Q}.\]
\end{corollary}

\section{The Steenrod reduced power operations for motivic cohomology}\label{SecMot}
One of the key roles in the proof of Theorem \ref{thm2p^k+2}, (2) is played by the Steenrod reduced power operations for motivic cohomology theory \cite{mazza2011lecture}.

Motivic cohomology is a functor from the category of smooth schemes over a base field (which we fix as $\mathbb{C}$) to that of bigraded abelian groups. For a smooth variety $X$  and an abelian group $R$, let $H^{*,*}(X;R)$ denote the motivic cohomology of $X$.

For an algebraic group $G$ and an integer $N\geq0$, the group $A^k_G$ for $k<N$ is by definition $A^k(U/G)$ where $U$ is an open subscheme of a representation such that $G$ acts freely on $U$. According to the discussion following Proposition \ref{Homotopy Invariance}, we may choose $U$ such that the quotient $U/G$ exists as a scheme. The motivic cohomology of $\mathbf{B}G$ is defined in a similar way as $A^*_G$. We write $H^{*,*}_G$ for $H^{*,*}(\mathbf{B}G;\Z)$, $H^{*,*}_G[p]$ for $H^{*,*}(\mathbf{B}G;\Z_{(p)})$, and $H^{*,*}_G(p)$ for $H^{*,*}(\mathbf{B}G;\Z/p)$. All the assertions in this section hold for $X=\mathbf{B}G$.

Motivic cohomology theory is a generalization of the theory of Chow groups in the sense of the following natural isomorphism:
\begin{equation}\label{eq:MotChow}
H^{2t,t}(X;R)\cong A^t(X)\otimes R.
\end{equation}
Over the field $\mathbb{C}$, the cycle class map also generalize to a natural map, usually called the realization map (3.11, \cite{voevodsky1999voevodosky}). For consistency we denote it by $\opn{cl}$:
\[\opn{cl}:H^{s,t}(X;R)\rightarrow H^s(X(\mathbb{C});R).\]

Motivic cohomology shares many nice properties with singular cohomology. For instance, a homomorphism of abelian groups $R_1\rightarrow R_2$ induces a natural transformation $H^{*,*}(-;R_1)\rightarrow H^{*,*}(-;R_2)$. Moreover, for a fixed $t$, the functor $H^{s,t}(X;-)$ is a $\delta$-functor. More precisely, let  $0\rightarrow R_0 \rightarrow R_1 \rightarrow R_2 \rightarrow0$ be a long exact sequence of abelian groups, then we have a long exact sequence
\begin{equation}\label{eq:mot les}
\cdots\rightarrow H^{s,t}(X;R_0)\rightarrow H^{s,t}(X;R_1)\rightarrow H^{s,t}(X;R_2)\xrightarrow{B}
H^{s+1,t}(X;R_0)\rightarrow\cdots,
\end{equation}
where $B$ is called the connecting (or the Bockstein) homomorphism, as in the case of singular cohomology.

In \cite{Vo}, Voevodsky defines cohomology operations for motivic cohomology theories with coefficients in $\Z/p$, for $p$ a prime number, similar to the Steenrod reduced power operations. When $p$ is odd, we have the operations
\[\scS^i: H^{s,t}(X;\Z/p)\rightarrow H^{s+2i(p-1),t+i(p-1)}(X;\Z/p),\hspace{2 mm} i\geq 0 \]
where $\scS^0$ is the identity, and
\[\beta: H^{s,t}(X;\Z/p)\rightarrow H^{s+1,t}(X;\Z/p).\]
which is the Bockstein homomorphism associated to the short exact sequence $0\rightarrow\Z/p\xrightarrow{\times p}\Z/p^2\rightarrow \Z/p\rightarrow0$, or equivalently, the Bockstein homomorphism associated to the short exact sequence $0\rightarrow\Z\xrightarrow{\times p}\Z\rightarrow \Z/p\rightarrow0$, composed with the mod $p$ reduction.

The operations $\scS^i$ and $\beta$ satisfy the Adem relations which are formally the same as in the case of singular cohomology. In particular, we have the following analog of \eqref{Adem1}:
\begin{equation}\label{eq:MotAdem1}
\mathscr{S}^{p^k}\beta\mathscr{S}^{p^{k-1}}=\beta\mathscr{S}^{p^k}\mathscr{S}^{p^{k-1}}.
\end{equation}

The motivic Steenrod operations and the Bockstein homomorphisms are compatible with their topological counterparts, in the sense of the following
\begin{proposition}[Voevodsky, \cite{voevodsky1999voevodosky}]\label{Steenrodcycleclass}
Let $\mathbb{K}=\mathbb{C}$ and $X$ be an algebraic variety over $\mathbb{C}$, we have the following commutative diagrams:
\begin{equation*}
\begin{tikzcd}
H^{s,t}(X;\mathbb{Z}/p)\arrow[r,"\mathscr{S}^i"]\arrow[d,"\opn{cl}"]
&H^{s+2i(p-1),t+i(p-1)}(X;\mathbb{Z}/p)\arrow[d,"\opn{cl}"]\\
H^s(X(\mathbb{C});\mathbb{Z}/p)\arrow[r,"\mathscr{P}^i"]&H^{s+2i(p-1)}(X(\mathbb{C});\mathbb{Z}/p)
\end{tikzcd}
\end{equation*}
and
\begin{equation*}
\begin{tikzcd}
H^{s,t}(X;\mathbb{Z}/p)\arrow[r,"B"]\arrow[d,"\opn{cl}"]
&H^{s+1,t}(X;\mathbb{Z})\arrow[d,"\opn{cl}"]\\
H^s(X(\mathbb{C});\mathbb{Z}/p)\arrow[r,"B"]&H^{s+1}(X(\mathbb{C});\mathbb{Z}).
\end{tikzcd}
\end{equation*}
\end{proposition}

The operations $\scS^i$ restrict to Chow rings in the sense of \eqref{eq:MotChow}:
\[\scS^i: A^t(X)\otimes\Z/p\rightarrow A^{t+i(p-1)}(X)\otimes\Z/p,\hspace{2 mm} i\geq 0.\]
Brosnan in \cite{Br} independently constructed cohomology operations for Chow rings, satisfying the axiomatic properties of the Steenrod power operations, including the Adem relations. It is unclear to the author whether they  agree with Voevodsky's definition. 


\section{The Chow ring and cohomology of $\mathbf{B}PGL_p$}\label{SecVistoli}
This section is a recollection of works of Vezzosi and Vistoli (\cite{Ve}, \cite{Vi}) on the Chow ring and integral cohomology of $\mathbf{B}PGL_p$ for $p$ an odd prime, together with a few further observations.

Recall that the Weyl group of $T(PGL_p)$, the maximal torus of $PGL_p$, is the permutation group $S_p$, which acts on $T(PGL_p)$ by permuting the diagonal entries. Elements in $A^*_{PGL_p}$ therefore restrict to  $(A^*_{T(PGL_p)})^{S_p}$, the subgroup of classes fixed by $S_p$. In Vistoli's paper \cite{Vi}, the Chow ring $A^*_{PGL_p}$ is given an $(A^*_{T(PGL_p)})^{S_p}$-algebra structure via a splitting injection
\[(A^*_{T(PGL_p)})^{S_p}\rightarrow A^*_{PGL_p}.\]
Vistoli determined $A^*_{PGL_p}$ in terms of the above $(A^*_{T(PGL_p)})^{S_p}$-algebra structure. We partially state his result as follows:
\begin{theorem}[Vistoli, \cite{Vi}, 2007]\label{VistoliPGLp}
The $(A^*_{T(PGL_p)})^{S_p}$-algebra $A^*_{PGL_p}$ is generated by an element $\rho_p\in A^{p+1}_{PGL_p}$ of additive order $p$.
\end{theorem}
\begin{remark}
See Section 3 of \cite{Vi} for the complete version of the theorem.
\end{remark}
The following result is Proposition 9.4 of \cite{Vi}.
\begin{proposition}[Vistoli, \cite{Vi},2007]\label{pro:Vistoli inj}
The homomorphisms
\[A^*_{PGL_p}\rightarrow A^*_{T(PGL_p)}\times A^*_{C_p\times\mu_p}\]
and
\[H^*_{PGL_p}\rightarrow H^*_{T(PGL_p)}\times H^*_{C_p\times\mu_p}\]
obtained from the embeddings $T(PGL_p)\hookrightarrow PGL_p$ and $C_p\times\mu_p\hookrightarrow PGL_p$ are injective.
\end{proposition}

To prove Theorem \ref{VistoliPGLp} and Proposition \ref{pro:Vistoli inj}, Vistoli considered two elements of $PGL_p$, represented respectively by the matrices

\[\begin{bmatrix}
    0 & \hdots & 0 & 1 \\
    1 & & & 0\\
    & \ddots  & &\vdots\\
    & & 1 & 0
    \end{bmatrix}
    \textrm{ and }
   \begin{bmatrix}
    \omega & & & & \\
    &\omega^2 & & & \\
    & & \ddots & &\\
    & & & \omega^{p-1} &\\
    & & & & 1
    \end{bmatrix},
\]
where $\omega$ is a $p$th root of unity. They generate two subgroups of $PGL_p$, both cyclic of order $p$, which we denote by $C_p$ and $\mu_p$, respectively. Furthermore, the two matrices commute up to a scalar $\omega$, which means they commute in $PGL_p$. Therefore we obtain an inclusion of algebraic groups $C_p\times\mu_p\hookrightarrow PGL_p$, which factors as
\begin{equation}\label{chainofincls}
C_p\times\mu_p\hookrightarrow C_p\ltimes T(PGL_p)\hookrightarrow S_p\ltimes T(PGL_p) \hookrightarrow PGL_p,
\end{equation}
where the two terms in the middle are the obvious semi-direct products.
\begin{remark}\label{choice of rho}
As pointed out by Vistoli (Remark 11.4, \cite{Vi}), the element $\rho_p$ depends on the choice of the $p$th root of unity $\omega$. Indeed, one readily verifies that, for a given choice of $\omega$ and the corresponding $\rho_p$, other choices of $\omega$, say $\omega'$, corresponds to $\lambda\rho_p$ for $\lambda$ running over $(\mathbb{Z}/p)^{\times}$.
\end{remark}
One readily verifies
\begin{proposition}\label{A(Cpmup)}
We have
\[A^*_{C_p\times\mu_p}\cong\mathbb{Z}[\xi,\eta]/(p\xi,p\eta),\]
where $\xi$ and $\eta$, both of degree $1$, are respectively the restrictions of the canonical generators of $A^*_{C_p}$ and $A^*_{\mu_p}$ via the projections.
\end{proposition}
Let $N_{C_p\times\mu_p}PGL_p$ be the normalizer of $C_p\times\mu_p$ in $PGL_p$. The quotient group $N_{C_p\times\mu_p}PGL_p/C_p\times\mu_p$ then acts on $A^*_{C_p\times\mu_p}$, and the image of the restriction $A^*_{PGL_p}\to A^*_{C_p\times\mu_p}$ is contained (and in fact equal to) the subgroup of invariance of this action.
Vistoli showed the following
\begin{proposition}[Vistoli, Proposition 5.4, \cite{Vi}, 2007]\label{q and r}
The quotient group
\[N_{C_p\times\mu_p}PGL_p/C_p\times\mu_p\]
is isomorphic to $SL_2(\mathbb{Z}/p)$. An element of $SL_2(\mathbb{Z}/p)$ acts on $A^*_{C_p\times\mu_p}$ by  extending its action on $A^1_{C_p\times\mu_p}\cong\Z/p\times\Z/p$ to a ring homomorphism.
Furthermore, the ring of invariants $(A^*_{C_p\times\mu_p})^{SL_2(\mathbb{Z}/p)}$ is generated by
\[q:=\xi^{p^2-p}+\eta^{p-1}(\xi^{p-1}-\eta^{p-1})^{p-1}\]
and
\[r:=\xi\eta(\xi^{p-1}-\eta^{p-1}).\]
\end{proposition}
Vistoli constructed the class $\rho_p$ by lifting $r$ successively via the restrictions associated to the chain of inclusions (\ref{chainofincls}). In other words, we have
\begin{equation}\label{eq:Vistoli Pro11.1}
A_{PGL_p}^{p+1}\rightarrow A_{C_p\times\mu_p}^{p+1},\hspace{5 mm} \rho_p\mapsto r,
\end{equation}
where the homomorphism is the obvious restriction. This is stated in Proposition 11.1 of Vistoli's paper \cite{Vi}.

Here we present two steps in the lifting process:
\begin{proposition}[Vistoli, Proposition 7.1 (d), \cite{Vi}, 2007]\label{VistoliInjection}
The ring homomorphisms
\[A^*_{C_p\ltimes T(PGL_p)}\rightarrow A^*_{T(PGL_p)}\times A^*_{C_p\times\mu_p}\]
and
\[H^*_{C_p\ltimes T(PGL_p)}\rightarrow H^*_{T(PGL_p)}\times H^*_{C_p\times\mu_p}\]
induced by the obvious restrictions are injective.
\end{proposition}

\begin{proposition}[Vistoli, Proposition 8.1, \cite{Vi}, 2007]\label{Vistoli(p-1)!}
The localized restriction homomorphism
\[(A^*_{S_p\ltimes T(PGL_p)})^W\otimes\mathbb{Z}[1/(p-1)!]\cong A^*_{C_p\ltimes T(PGL_p)}\otimes\mathbb{Z}[1/(p-1)!]\]
is an isomorphism. Here $W$ is the Weyl group of $C_p\ltimes T(PGL_p)$ in $S_p\ltimes T(PGL_p)$.
\end{proposition}

The singular cohomology of $C_p\times\mu_p$ also plays an important role.
\begin{proposition}\label{H(Cpmup)}
We have
\[H^*_{C_p\times\mu_p}\cong\mathbb{Z}[\xi,\eta,\zeta]/(p\xi,p\eta,p\zeta,\zeta^2)\]
where $\xi$ and $\eta$ are of degree $2$, and are the images of elements in $A^*_{C_p\times\mu_p}$ denoted by the same letters, via the cycle class map, whereas $x_1$ is of degree $3$.
\end{proposition}
\begin{proof}
It is standard homological algebra that we have
\[H^*(\mathbf{B}C_p;\mathbb{Z}/p)=\Lambda_{\mathbb{Z}/p}(a)\otimes\mathbb{Z}/p[\bar{\xi}]\]
and
\[H^*(\mathbf{B}\mu_p;\mathbb{Z}/p)=\Lambda_{\mathbb{Z}/p}(b)\otimes\mathbb{Z}/p[\bar{\eta}]\]
where $a$ and $b$ are of degree $1$, and $\Lambda_{\mathbb{Z}/p}$ means exterior algebra over $\mathbb{Z}/p$, such that the Bockstein homomorphism satisfies
\[B(a)=\xi\textrm{ and }B(b)=\eta.\]
Furthermore, we have
\[\mathscr{P}^1(a)=0\textrm{ and }\mathscr{P}^1(a)=0\]
following from the degree axiom of the Steenrod reduced power operations (Steenrod and Epstein, \cite{St}). Since $\bar{\xi}$ and $\bar{\eta}$ are of degree $2$, another axiom asserts
\[\mathscr{P}^1(\bar{\xi})=\bar{\xi}^p\textrm{ and }\mathscr{P}^1(\bar{\eta})=\bar{\eta}^p.\]
The isomorphism
\[H^*_{C_p\times\mu_p}\cong\mathbb{Z}[\xi,\eta,\zeta]/(p\xi,p\eta,p\zeta,\zeta^2)\]
follows from the K{\"u}nneth formula. Indeed, we may define $\zeta$ as the integral lift of $\bar{\xi}b-a\bar{\eta}$.
\end{proof}

Let an integral cohomology class with an overhead bar denote the mod $p$ reduction of this class. Using Cartan's formula for the Steenrod reduced power operations, we define
\begin{equation*}
\begin{split}
&r=B\mathscr{P}^1(\bar{\zeta})=B\mathscr{P}^1(\bar{\xi}b-a\bar{\eta})=B[\mathscr{P}^1(\bar{\xi})b-a
\mathscr{P}^1(\bar{\eta})]\\
=&B(\bar{\xi}^pb-a\bar{\eta}^p)={\xi}^p\eta-\xi{\eta}^p,
\end{split}
\end{equation*}
which is the image of $r\in A^{p+1}_{PGL_n}$ via the cycle class map. More generally for $k\geq0$ we have
\[r_k:=\xi\eta(\xi^{p^{k+1}-1}-\eta^{p^{k+1}-1})=B\mathscr{P}^k\mathscr{P}^{k-1}\cdots\mathscr{P}^1(\bar{\zeta}),\]
where $B$ is the connecting homomorphism $H^*(-;\mathbb{Z}/p)\rightarrow H^{*+1}(-;\mathbb{Z})$ and $\bar{\zeta}$ is the mod $p$ reduction of $\zeta$. By definition we have $r_0=r$. By Corollary \ref{SteenrodOperation2} and Proposition \ref{H(Cpmup)}, we obtain the following
\begin{corollary}\label{Steenrodpowers}
For $k>0$, we have
\[\bar{r}_k:=\mathscr{P}^{p^k}(\bar{r}_{k-1}),\]
in $H^*(\mathbf{B}(C_p\times\mu_p);\mathbb{Z}/p)$, and similarly we have
\[\bar{r}_k:=\mathscr{S}^{p^k}(\bar{r}_{k-1})\]
in $A^*_{C_p\times\mu_p}\otimes\mathbb{Z}/p$.
\end{corollary}

This leads to part (1) of Theorem \ref{thm2p^k+2}, which we record as follows.

\begin{proposition}[(1) of Theorem \ref{thm2p^k+2}]\label{pro:thm part1}
For $p$ an odd prime and $p|n$, we have $0\neq y_{p,k}\in H^{2(p^{k+1}+1)}_{PGL_n}$ for all $k\geq 0$.
\end{proposition}
\begin{proof}
Consider the composition of inclusions of algebraic groups
\begin{equation}\label{eq:inclusions}
\theta:C_p\times\mu_p\subset PGL_p\xrightarrow{\Delta} PGL_n,
\end{equation}
which induces the restriction $H^*_{PGL_n}\rightarrow H^*_{C_p\times\mu_p}$. Here $\Delta$ is the diagonal inclusion.
It follows from Proposition \ref{SteenrodOperation1} and Proposition \ref{H(Cpmup)} that the restriction takes $y_{p,k}$ to $r_k\neq0$, and we conclude.
\end{proof}

The following corollary plays an important role in the proof of Theorem \ref{thm2p^k+2}. As in Section \ref{SecSpecSeq}, $H^*_G[p]$ denotes $H^*_G$ localized at $p$.
\begin{corollary}\label{cor:H^2p+2 direct sum 2}
For $p$ an odd prime and $p|n$, the inclusion of the maximal torus $\lambda:T(PGL_n)\rightarrow PGL_n$ and the map $\theta:C_p\times\mu_p\rightarrow PGL_n$ give an isomorphism
\[H^{2(p+1)}_{PGL_n}[p]\cong (H^{2(p+1)}_{T(PGL_n)}[p])^{S_n}\oplus (H^{2(p+1)}_{C_p\times\mu_p})^{SL_2(\Z/p)},\]
or equivalently
\[H^{2(p+1)}_{PGL_n}[p]\cong (H^{2(p+1)}_{T(PGL_n)}[p])^{S_n}\oplus (r).\]
The map $\lambda^{2(p+1)}:H^{2(p+1)}_{PGL_n}\rightarrow H^{2(p+1)}_{T(PGL_n)}$ is a split epimorphism with right inverse $\phi$ satisfying $\theta^{2(p+1)}\phi=0$.
\end{corollary}
\begin{proof}
It is verified in the proof of Proposition \ref{pro:thm part1} that we have $\theta^*(y_{p,0})=r_0$, giving an monomorphism
\[H^{2(p+1)}(K(\Z,3);\Z_{(p)})\hookrightarrow H^{2(p+1)}_{C_p\times\mu_p}[p].\]
The rest follows Corollary \ref{cor:H^2p+2 direct sum}.
\end{proof}

The map $\theta$ detects the non-vanishing of $p$-torsion classes $y_{p,I}\in H^*_{PGL_n}$ for some $I$ satisfying $l(I)>1$. The following serves as a complement of Theorem \ref{thmypI}.
\begin{corollary}\label{cor:ypI}
Let $I=(0,1)$. Then for an odd prime number $p$ and $n$ satisfying $p|n$, the class $y_{p,I}\in H^*_{PGL_n}$ is nontrivial.
\end{corollary}
\begin{proof}
By Proposition \ref{SteenrodOperation0} we have
\[y_{p,I}=B(x_{p,1}x_{p,0})=B(\scP^p\scP^1(\bar{x}_1)\cdot\scP^1(\bar{x}_1)),\]
and by Proposition \ref{H(Cpmup)} we have
\begin{equation*}
\begin{split}
&\theta^*(y_{p,I})=\theta^*(B(\scP^p\scP^1(\bar{x}_1)\cdot\scP^1(\bar{x}_1)))\\
=&B(\scP^p\scP^1(\bar{\zeta})\cdot\scP^1(\bar{\zeta}))\\
=&B(\scP^p\scP^1(\bar{\xi}b-a\bar{\eta})\cdot\scP^1(\bar{\xi}b-a\bar{\eta}))\\
=&-\zeta(\xi^p\eta^{p^2}+\xi^{p^2}\eta^p)\neq0.
\end{split}
\end{equation*}
\end{proof}
The essential ingredient of Vistoli \cite{Vi} and Vezzosi \cite{Ve} is the stratification method. We adopt the following notation of a stratification from \cite{Vi}: given an algebraic group $G$ and a (complex) $G$-representation $V$, a stratification of $V$ is a series of Zariski locally closed $G$-equivariant sub-varieties of $V$, say $V_1,V_2,\cdots,V_t$ such that each $V_{\leq i}:=\bigcap_{j\leq i}V_j$ is Zariski open in $V$, each $V_i$ is closed in $V_{\leq i}$, and $V_t=V\backslash\{0\}$. If we can obtain generators for some $A^*_G(V_{\leq i})$, then by induction on $i$ and using the localization sequence
\[A^*_G(V_{i+1})\rightarrow A^*_G(V_{\leq i+1})\rightarrow A^*_G(V_{\leq i})\rightarrow 0\]
we obtain generators of $A^*_G(V)=A^*_G$.

One of the advantages of working with stratifications is that it may enable us to simplify the group $G$. For example, for any integer $n>1$, consider the $PGL_n$-representation $sl_n$ of trace-zero $n\times n$ matrices on which $PGL_n$ acts by conjugation. Similarly, we have a representation $D_n$ of the group $\Gamma_n:=S_n\ltimes T(PGL_n)$, defined as the $n\times n$ diagonal trace-zero matrices, on which $\Gamma_n$ acts by conjugation. Let $sl_n^*$ (resp. $D_n^*$) be the open subvariety of $sl_n$ (resp. $D_n$) of matrices with $n$ distinct eigenvalues. Then we have the following

\begin{proposition}[Vezzosi, Proposition 3.1, \cite{Ve}, 2000]\label{Vezzosi}
The composite of natural maps
\[A^*_{PGL_n}(sl_n^*)\rightarrow A^*_{\Gamma_n}(sl_n^*)\rightarrow A^*_{\Gamma_n}(D_n^*)\]
is a ring isomorphism.
\end{proposition}
\begin{remark}
Vezzosi stated this proposition only for $n=3$, but his proof works for any $n>1$. Indeed, his proof is essentially an application of Proposition \ref{induced space}, taking $G=PGL_n$, $H=\Gamma_n$ and $X=D_n^*$.
\end{remark}
Based on Proposition \ref{Vezzosi}, Vistoli proved the following more refined result when $n$ is an odd prime $p$.
\begin{proposition}[Vistoli, Proposition 10.1, \cite{Vi}, 2007]
The restriction $A^*_{PGL_p}\rightarrow A^*_{\Gamma_p}$ sends the Chern class $c_{p^2-1}(sl_p)$ into the ideal generated by the Chern class $c_{p-1}(D_p)$ and the induced map
\[A^*_{PGL_p}/c_{p^2-1}(sl_p)\rightarrow A^*_{\Gamma_p}/c_{p-1}(D_p)\]
is a ring isomorphism after reverting $(p-1)!$.
\end{proposition}
Indeed, via the localization sequences
\[A^*_{PGL_p}\xrightarrow{c_{p^2-1}(sl_p)}A^*_{PGL_p}\rightarrow A^*_{PGL_p}(sl_p\backslash\{0\})\rightarrow0\]
and
\[A^*_{\Gamma_p}\xrightarrow{c_{p-1}(D_p)}A^*_{\Gamma_p}\rightarrow A^*_{\Gamma_p}(D_p\backslash\{0\})\rightarrow0\]
we make the following identifications:
\begin{equation}\label{moduloChernclass}
\begin{split}
A^*_{PGL_p}/c_{p^2-1}(sl_p)&\cong A^*_{PGL_p}(sl_p\backslash\{0\}),\\
A^*_{\Gamma_p}/c_{p-1}(D_p)&\cong A^*_{\Gamma_p}(D_p\backslash\{0\}).
\end{split}
\end{equation}
Vistoli then showed the following
\begin{lemma}[Vistoli, within the proof of Proposition 10.1, \cite{Vi}, 2007]
All arrows (obvious restriction maps) in the following diagram are ring isomorphisms:
\begin{equation*}
\begin{tikzcd}
A^*_{PGL_p}(sl_p\backslash\{0\})\arrow{r}\arrow{d}&A^*_{PGL_p}(sl_p^*)\arrow{d}\\
A^*_{\Gamma_p}(D_p\backslash\{0\})\arrow{r}&A^*_{\Gamma_p}(D_p^*).
\end{tikzcd}
\end{equation*}
\end{lemma}

The following lemma is also due to Vistoli, though not explicitly stated.
\begin{lemma}[Vistoli, within the proof of Lemma 10.2, \cite{Vi}, 2007]\label{stratificationCpmup}
Suppose that $W$ is a representation of $C_p\ltimes T(PGL_p)$, and $U$
an open subset of $W\backslash\{0\}$. Assume that
\begin{enumerate}[label=(\alph*)]
    \item the restriction of W to $C_p\times\mu_p$ splits as a direct sum of $1$-dimensional representations $W=L_1\oplus L_2\oplus\cdots\oplus L_r$, in such a way that the characters $C_p\times\mu_p\rightarrow\mathbb{C}^{\times}$ of $L_i$'s are all distinct, and each $L_i\backslash\{0\}$ is contained in $U$, and
    \item $U$ contains a point that is fixed under $T(PGL_p)$.
\end{enumerate}
Then the restriction homomorphism $A^*_{C_p\times\mu_p}(W\backslash\{0\})\rightarrow A^*_{C_p\times\mu_p}(U)$ is an isomorphism.
\end{lemma}
\begin{remark}
Vistoli stated in Lemma 10.2, \cite{Vi}, that under these conditions, the restriction homomorphism $A^*_{C_p\ltimes T(PGL_p)}(W\backslash\{0\})\rightarrow A^*_{C_p\ltimes T(PGL_p)}(U)$ is an isomorphism.
\end{remark}
Lemma \ref{stratificationCpmup} has the following consequence that will be needed later. Recall the $\Gamma_n$-representation $D_n$ and its open subvariety $D_n^*$. When $p|n$, consider $D_n$ as a $C_p\times\mu_p$-representation via the composition
\[C_p\times\mu_p\hookrightarrow C_p\ltimes T(PGL_p)\hookrightarrow\Gamma_p\xrightarrow{\Delta}\Gamma_n\]
where $\Delta$ is the diagonal homomorphism.
\begin{corollary}\label{restriction to p}
For $p|n$, the restriction homomorphism
\[A^*_{C_p\times\mu_p}/c_{n-1}(D_n)\cong A^*_{C_p\times\mu_p}(D_n\backslash\{0\})\rightarrow A^*_{C_p\times\mu_p}(D_n^*)\]
is an isomorphism. Here $c_{n-1}(D_n)$ is the $(n-1)$th Chern class of the $C_p\times\mu_p$-representation $D_n$.
\end{corollary}
\begin{proof}
Take $W=D_n$, $U=D_n^*$, and apply Lemma \ref{stratificationCpmup}.
\end{proof}

\section{The subgroups of $PGL_n$ of diagonal block matrices}\label{SecBlock}
Let $W=(n_1,\cdots, n_r)$ be a sequence of non-negative integers such that $n=\sum_i n_i$. Let $GL_W=\prod_i GL_{n_i}$, viewed as a subgroup of $GL_n$ via the diagonal inclusion, and let $PGL_W$ be the subgroup of $PGL_n$ of images of $GL_W$ via the canonical projection.

In \cite{Ve}, Vezzosi proved the following
\begin{proposition}[Vezzosi, Corollary 2.4, \cite{Ve}, 2000]\label{Vezzosi n torsion}
All torsion classes in $A^*_{PGL_n}$ are $n$-torsion.
\end{proposition}
In the rest of this section we prove a generalization of Proposition \ref{Vezzosi n torsion} as follows:
\begin{proposition}\label{BPGL_W}
Let $W=(n_1,\cdots, n_r)$ be an ordered partition of $n$, and let $\opn{gcd}(W)$ be the greatest common divisor of $n_1, \cdots, n_r$. Then all torsion classes in $A^*_{PGL_W}$ are $\opn{gcd}(W)$-torsion.
\end{proposition}

Consider the $\mathbb{C}$-algebra of $n\times n$ matrices $M_n(\mathbb{C})$. For $W$ as above, Let $M_W(\mathbb{C})$ be the sub-algebra of $M_n(\mathbb{C})$ of diagonal block matrices of the form
\begin{equation*}
   \begin{bmatrix}
       A_1 & 0 & \hdots & 0\\
       0   & A_2 & \ddots    & \vdots\\
       \vdots   & \ddots   & \ddots & 0\\
       0   & \hdots & 0 &   A_r
   \end{bmatrix}
\end{equation*}
such that $A_i$ is an $n_i\times n_i$ matrix. 



The inclusion $PGL_W\to PGL_n$ associates every $PGL_W$-torsor $\xi: E\rightarrow X$ to a rank $n$ Azumaya algebra, i.e., an {\'e}tale sheaf of algebras of $n\times n$ matrices, over the base scheme $X$,  which we denote by $A(\xi): A(E)\rightarrow X$. By construction $A(\xi)$ has a sub-algebra, at each fiber giving rise to the inclusion $M_W(\mathbb{C})\hookrightarrow M_n(\mathbb{C})$. We denote it by $A_W(\xi): A_W(E)\rightarrow X$. The inclusion of the $i$th diagonal block $M_{n_i}(\mathbb{C})\to M_W(\mathbb{C})$ gives rise to subalgebras of $A_W(\xi)$ which we denote by $A_i(\xi): A_i(E)\to X$.


An essential ingredient of Vezzosi's proof of Proposition \ref{Vezzosi n torsion} is the modified push-forward. Let $f: Y\rightarrow X$ be a smooth proper morphism of relative dimension $r$. Then we have the modified push-forward
\[f_\#: A^*(Y)\rightarrow A^*(X), a\mapsto f_*(c_r(T_f)a),\]
where $f_*$ is the usual push-forward, $c_r(T_f)$ the top Chern class of the relative tangent bundle of $f$.
\begin{lemma}\label{push-pull}
Let $f: Y\rightarrow X$ be as above, and let $F$ be the fiber of $f$ over a non-singular point. Then we have
\[f_\#f^*(a)=\chi(F)a\]
where $\chi(F)$ is the Eular characteristic of $F$.
\end{lemma}
\begin{proof}
By the projection formula of $f_*$ and $f^*$, it suffices to $f_*(c_r(T_f))=\chi(F)$. Since we have $f_*(c_r(F))=\chi(F)$, it suffices to show that the diagram
\begin{equation*}
    \begin{tikzcd}
    A^*(F)\arrow[d,"f_*"]& A^*(Y)\arrow[l]\arrow[d,"f_*"]\\
    A^*(\opn{spec}\mathbb{C})& A^*(X)\arrow[l]
    \end{tikzcd}
\end{equation*}
commutes, where the horizontal arrows are the pullbacks induced by the obvious inclusions. But this is a standard argument which can be found in, for instance, 41.15, \cite{Sta} .
\end{proof}

\begin{proof}[Proof of Proposition \ref{BPGL_W}]
Consider the restriction of the quotient map $\pi: GL_W\rightarrow PGL_W$ restricted to the canonical maximal tori, $\pi_T:T(GL_W)\rightarrow T(PGL_W)$. A routine computation shows that the restriction $\pi_T^*$ is injective. On the other hand, $PGL_W$ is reductive since the adjoint representation of $PGL_W$ is faithful and is a direct sum of irreducible representations. By Proposition \ref{Maximal torus}, the subgroup of $A^*_{PGL_W}$ of torsion classes is the kernel of the restriction $A^*_{PGL_W}\rightarrow A^*_{GL_W}$. On the other hand, given a class $\alpha\in A^*_{PGL_W}$ , the class $\pi^*(\alpha)$, regarded as a universal characteristic class, is defined by $\pi^*(a)(\eta)=a(\bar{\eta})$, where $\eta$ is a $GL_W$-torsor and $\bar{\eta}$ is the projective bundle associated to $\eta$.

Therefore, it suffices to show the following: If $\alpha\in A^*_{PGL_W}$ such that $\alpha(\bar{\eta})=0$ for all $GL_W$-torsors $\eta$, then we have $n_i\alpha(\xi)=0$ for all $i$ and all $PGL_W$-torsors $\xi$. We fix such an $\alpha$.


Let $\xi: E\rightarrow X$ be a $PGL_n$-torsor over $X$ which lifts to a $PGL_W$-torsor via the inclusion $PGL_W\to PGL_n$. Let $g: \mathbb{P}(E)\rightarrow X$ be the associated Severi-Brauer variety. Then the Azumaya algebra $g^*(A(\xi))$ is isomorphic to $\opn{End}(\eta)$ for some $n$ dimensional vector bundle over $\mathbb{P}(E)$. In other words, the pullback $PGL_n$-torsor $g^*(\xi)$, regarded as a $PGL_n$-torsor, lifts to a $GL_n$-torsor via the quotient map. We show that the $PGL_n$-torsor $g^*(\xi)$ lifts to a $GL_W$-torsor.

Without loss of generality, suppose that descent data for the lift of $g^*(\xi)$ to both a $PGL_W$-torsor and a $GL_n$-torsor are given by the same cover $\{U_i\}_i$. We use the notation $U_{i_1i_2\cdots i_r}$ for the intersection of $U_{i_i},\cdots, U_{i_r}$. Let $\{\varphi_{ij}:U_{ij}\to GL_n\}$ and $\{\psi_{ij}:U_{ij}\to GL_W,\hspace{2 mm}t_{ijk}:U_{ijk}\to\mathbb{C}^{\times}\}$ be two choices of descent data for the lift of $g^*(\xi)$ to the $GL_n$-torsor and the $PGL_W$-torsor, respectively. Here $t_{ijk}$ is characterised by
\begin{equation}\label{eq:descent1}
\psi_{ij}(u)\psi_{jk}(u)=t_{ijk}(u)\psi_{ik}(u),\hspace{2 mm} u\in U_{ijk}.
\end{equation}
Furthermore, let $\{\lambda: U_{ij}\to GL_n,\hspace{2 mm}s_{jkl}:U_{jkl}\to\mathbb{C}^{\times}\}$ be a morphism between the two set of descent data above. More precisely, we have
\begin{equation}\label{eq:descent2}
\varphi_{ij}(u)\lambda_{jk}(u)=\lambda_{il}(u),\hspace{2 mm}u\in U_{ijk},
\end{equation}
and
\begin{equation}\label{eq:descent3}
\lambda_{jk}(u)\psi_{kl}(u)=\lambda_{jl}(u)s_{jkl}(u),\hspace{2 mm}u\in U_{jkl}.
\end{equation}
It follows from \eqref{eq:descent2} and \eqref{eq:descent3} that $s_{jkl}$ is independent of $j$, and we will write $s_{kl}$ instead. Therefore we have
\begin{equation}\tag{\ref{eq:descent3}'}\label{eq:descent3'}
\lambda_{jk}(u)\psi_{kl}(u)=\lambda_{jl}(u)s_{kl}(u),\hspace{2 mm}u\in U_{jkl}.
\end{equation}
By \eqref{eq:descent1}, \eqref{eq:descent2} and \eqref{eq:descent3'}, we have
\begin{equation}\label{eq:descent4}
s_{kl}(u)s_{lm}(u)=t_{klm}(u)s_{km}(u), \hspace{2 mm}u\in U_{klm}.
\end{equation}
Define $\psi_{kl}':U_{kl}\rightarrow GL_W$ by $\psi_{kl}'(u)=s_{kl}^{-1}(u)\psi_{kl}(u)$. Then it follows from \eqref{eq:descent1} and \eqref{eq:descent4} that the functions $\{\psi_{kl}'\}$ give a choice of descent data of a $GL_W$-torsor that reduces to the $PGL_W$-torsor $g^*(\xi)$ via the quotient map $GL_W\to PGL_W$. By our assumption for $\alpha$ we have

\begin{equation}\label{g*E}
\alpha(g^*{\xi})=0.
\end{equation}

On the other hand, the canonical projection $PGL_W\rightarrow PGL_{n_i}$ gives rise to an induced $PGL_{n_i}$-torsor $\xi:E_i\rightarrow X$, and the associated Severi-Brauer variety:
\[g_i:\mathbb{P}(E_i)\rightarrow X.\]
Furthermore, $g_i$ factors as follows:
\[\mathbb{P}(E_i)\xrightarrow{h_i}\mathbb{P}(E)\xrightarrow{g}X\]
where $h_i$ is given by the inclusion
\[\mathbb{P}^{n_i-1}\rightarrow \mathbb{P}^{n-1}, [a_0,\cdots,a_{n_i-1}]\mapsto [0,\cdots,a_0,\cdots,a_{n_i-1},\cdots,0]\]
sending the coordinates to the $n_{i-1}+1,\cdots,n_i$th entries. (By convention $n_0=0$.) This fiber-wise inclusion passes to the total spaces since it is $PGL_W$-equivariant. Applying (\ref{g*E}) and Lemma \ref{push-pull}, we have
\[n_i\alpha(\xi)=\chi(\mathbb{P}^{n_i-1})\alpha(\xi)=(g_i)_{\#}g_i^*(\alpha(\xi))=(g_i)_{\#}h_i^*g^*(\alpha(\xi))=0,\]
and we conclude.
\end{proof}


\section{The permutation groups and their double quotients}\label{SecPer}
This section is a technical prerequisite for the construction of the class $\rho_{p,0}\in A^{p+1}_{PGL_n}$, to be presented in Section \ref{SecRho}. Throughout this section, $n$ will be a positive integer and $p$ an odd prime divisor of $n$.

We take $W=(n_1,\cdots,n_r)$ such that $n=\sum_i n_i$, and let $n_0=0$ as in Section \ref{SecBlock}. Then the canonical actions of $S_{n_i}$ on the sets $\{n_{i-1}+1,\cdots,n_i\}$ identifies $S_W=S_{n_1,\cdots,n_r}:=S_{n_1}\times\cdots\times S_{n_r}$ as a subgroup of $S_n$.

\begin{remark}\label{permutationmatrices}
In the context of this paper, particularly this and the next section, it is sometimes helpful to regard the permutation group $S_n$ as the subgroup of $GL_n$ of permutation matrices acting on column vectors $e_i$ that forms the canonical basis of $\mathbb{C}$:
\[e_i=(0,\cdots,0,1,0,\cdots,0)^T\textrm{ (with $i$th entry $1$)}\]
from the left. (Here ``T'' means transpose, of course.) More precisely, If $s\in S_n$, then as a permutation matrix, $s$ satisfies
\begin{equation}\label{permutation matrix}
se_j=e_{s(j)}.
\end{equation}
Yet in other words, the $j$th column of $s$ is $e_{s(j)}$. We will freely let elements in $S_n$ acts either on the vectors $e_i$'s, or numbers $1\leq i\leq n$, without further explanation.

For example, the subgroup $S_{p,n-p}$ of $S_n$ consists of matrices of the form
\begin{equation*}
\begin{bmatrix}
    A &  0  \\
    0 &  B
\end{bmatrix},
\end{equation*}
where $A$ and $B$ are permutation matrices of dimensions $p$ and $n-p$, respectively.
\end{remark}

The following is a routine computation.
\begin{lemma}\label{congruence}
Let $W=(n_1, \cdots, n_r)$ be as above and $n_0=0$ by convention. For $s\in S_n$, let $F_i$ be the set of vectors of the $(n_{i-1}+1)$th to the $n_i$th columns of $s$. Then we have
\begin{equation}\label{congruenceEquation}
sS_Ws^{-1}=\{t\in S_n|t(F_i)=F_i,\textrm{ for all }1\leq i\leq r.\}.
\end{equation}
\end{lemma}

We proceed to study double quotients of the form $S_W\backslash S_n/S_{p,n-p}$.

\begin{lemma}\label{double quotient}
\begin{enumerate}
\item
The left quotient set $S_n/S_{p,n-p}$ is in 1-1 correspondence to linear subspaces of dimension $p$ in $\mathbb{C}^n$ spanned by vectors $e_i$, or equivalently, subsets of $\{e_i\}_{i=1}^n$ consisting of $p$ elements.

\item
Let $W=(n_1,\cdots,n_r)$ be as above. The orbits of the canonical action of $S_W$ on the left quotient set $S_n/S_{p,n-p}$ are of the form $\mathfrak{O}_K$ where $K=(k_1,\cdots,k_r)$ is a sequence of non-negative integers summing up to $p$ and satisfies $k_i\leq n_i$. The elements of the orbit $\mathfrak{O}_K$ are subsets of $\{e_j\}_{j=1}^n$ containing exactly $k_i$ elements of the form $e_j$ for $n_{i-1}<j\leq n_i$. Therefore, we have
\[S_W\backslash S_n/S_{p,n-p}=\{\mathfrak{O}_K | K=(k_1,\cdots,k_r), 0\leq k_i\leq n_i, \sum_i k_i=p\}.\]
\end{enumerate}
\end{lemma}
\begin{proof}
In view of Remark \ref{permutationmatrices}, the left quotient $S_n/S_{p,n-p}$ is very similar to the construction of Grassmannians in, for example, Chapter 4 of Switzer \cite{Sw}. Indeed, one readily verifies that two permutation matrices $A,B\in S_n$ represents the same coset in $S_n/S_{p,n-p}$ if and only if the first $p$ columns of $A$ and the first $p$ columns of $B$ span the same set of vectors in $\mathbb{C}^n$, and it follows that the left quotient set $S_n/S_{p,n-p}$ is in 1-1 correspondence to sets of vectors of the form $e_i$ having $p$ elements.

For the assertion on the double quotient, simply observe that the canonical left action of $S_W$ on $S_n/S_{p,n-p}$ by permuting the $j$th rows for $n_{i-1}<j\leq n_i$ and $1\leq i\leq r$. More precisely, let $L$ be an $n\times p$ matrix of $p$ column vectors of the form $e_i$, and let $[L]$ be the set of column vectors of $L$. Then we regard $[L]$ as an element of $S_n/S_{p,n-p}$. Let $\sigma\in S_W$. Then $\sigma$ is represented by a $n\times n$ permutation matrix of the form
\begin{equation}\label{eq:diagsigma}
\sigma=\begin{bmatrix}
\sigma_1 &  & \\
  & \ddots & \\
  &    & \sigma_r
\end{bmatrix}
\end{equation}
where $\sigma_k\in S_k$ permutes $e_{n_{i-1}+1},\cdots,e_{n_i}$ by left multiplication, for $1\leq i\leq r$, and $n_0=0$.

Suppose in $[L]$ there are exactly $k_i$ element of the form $e_j$ for $n_{i-1}+1\leq j\leq n_i$. Then the same is true for $[\sigma L]$. Conversely, for $[L],[L']$ both having exactly $k_i$ element of the form $e_j$ for $n_{i-1}+1\leq j\leq n_i$, we may choose $\sigma_i\in S_{k_i}$ permuting $e_{n_{i-1}+1},\cdots,e_{n_i}$ in such a way that we have $\sigma L=L'$, where $\sigma$ is as in \eqref{eq:diagsigma}.
\end{proof}

Let $W=(n_1,\cdots,n_r)$ be as above. In view of Lemma \ref{double quotient}, we identify the double quotient $S_W\backslash S_n/S_{p,n-p}$ with the orbits $\mathfrak{O}_K$. We need the following notation for the next lemma. Suppose $F$ is a subset of $\{1,\cdots,n\}$ and $W$ is as above. Let $m_i$ be the cardinality of the set $F\bigcap \{n_{i-1}+1,\cdots,n_i\}$, and let
\begin{equation}\label{W slash F}
W/F:=(m_1,n_1-m_1,m_2,n_2-m_2,\cdots,m_r,n_r-m_r).
\end{equation}

\begin{lemma}\label{index}
Let $s\in S_n$ of which the first $p$ columns form the set $e(F)=\{e_i|i\in F\}$. Then we have
\begin{equation}\label{Intersection}
sS_{p,n-p}s^{-1}\bigcap S_W=\{t\in S_W| t(e(F))=e(F)\}.
\end{equation}
In particular, we have a group isomorphism
\[sS_{p,n-p}s^{-1}\bigcap S_W\cong S_{W/F}.\]
\end{lemma}
\begin{proof}
Equation (\ref{Intersection}) follows immediately from Lemma \ref{congruence}. For the second statement, it suffices to observe that $sS_{p,n-p}s^{-1}\bigcap S_W$ is the subgroup of $S_n$ which acts separately on the sets $\{e_{n_{i-1}+1},\cdots,e_{n_i}\}\bigcap e(F)$ and $\{e_{n_{i-1}+1},\cdots,e_{n_i}\}-e(F)$, such that on each of these subsets the action is transitive. After perhaps re-ordering $F$, this yields $S_{W/F}$, and we conclude.
\end{proof}



\section{The torsion classes $\rho_{p,k}$ in $A^*_{PGL_n}$}\label{SecRho}
We are finally prepared to construct the promised torsion classes $\rho_{p.k}\in A^*_{PGL_n}$. When $n=p$, the class $\rho_{p,0}$ is simply the $\rho_p$ in Theorem \ref{VistoliPGLp}.

Throughout this section, $p$ will be an odd prime and $n$ a positive integer such that $p|n$.

As in \cite{Vi}, we first take $\Gamma_n=S_n\ltimes T(PGL_n)$ as an avatar of $PGL_n$ and then apply the localization sequence of Chow groups to obtain the desired result for $PGL_n$. Since we consider $\Gamma_n$ as a subgroup of $PGL_n$ in the obvious way, elements in it are represented by $n\times n$ matrices such that in each row and column there is exactly one nonzero entry.

Recall the subgroups $S_W$ of $S_n$ defined in Section \ref{SecPer}. In particular we consider $S_{(p)}$ where we have
\[(p):=(p,p,\cdots,p)\]
the series with $n/p$ copies of $p$.

In the obvious sense we take the subgroup $\Gamma_W:=S_W\ltimes T(PGL_n)$ of $\Gamma_n$. Then elements in $\Gamma_{(p)}$ and $\Gamma_{p,n-p}$ are represented respectively by block matrices of the forms
\begin{equation}\label{(p) blocks}
\begin{bmatrix}
    A_1 & &\\
    & \ddots & \\
    & &     A_{n/p}
\end{bmatrix}
\end{equation}
and
\begin{equation}\label{p,n-p blocks}
    \begin{bmatrix}
    A & 0 \\
    0 & B
    \end{bmatrix}
\end{equation}
with $A$ ($A_i$), $B$ square matrices of dimensions $p$ and $n-p$, respectively. We have the following group homomorphisms
\begin{equation}\label{p,n-p blocks to p}
 \Gamma_p\rightarrow\Gamma_{p,n-p} (\textrm{or }\Gamma_{(p)}), \quad
 A\mapsto
  \begin{bmatrix}
    A &  &  \\
    &\ddots  & \\
    &   &   A
    \end{bmatrix}
\end{equation}
and
\begin{equation}\label{eq:p,n-p to p}
  \Gamma_{p,n-p}\rightarrow\Gamma_p, \quad
  \begin{bmatrix}
    A & 0 \\
    0 & B
    \end{bmatrix}
  \mapsto A,
\end{equation}
and
\[
  \Gamma_{(p)}\rightarrow\Gamma_p, \quad
  \begin{bmatrix}
    A_1 & &\\
     &\ddots &\\
     &  &   A_{n/p}
    \end{bmatrix}
  \mapsto A_1,
\]
where $A$, $B$ are as in (\ref{p,n-p blocks}). One readily observes that the composite of the homomorphisms above gives the identity of $\Gamma_p$, which allows us to define $p$-torsion classes $\rho'_p\in A^{p+1}_{\Gamma_p}$ and $\rho''_p\in A^{p+1}_{\Gamma_{p,n-p}}$ as follows:
\begin{equation}\label{eq:def rho'}
\begin{split}
A^{p+1}_{PGL_p}\rightarrow &A^{p+1}_{\Gamma_p}\leftarrow A^{p+1}_{\Gamma_{p,n-p}},\\
\rho_p\mapsto &\rho'_p\testleft\rho''_p,
\end{split}
\end{equation}
where $\rho_p$ is given in Theorem \ref{VistoliPGLp}, and the second arrow is induced by the homomorphism \eqref{eq:p,n-p to p}.

In particular, let $\rho'_p\in A^{p+1}_{\Gamma_p}$ be the restriction of $\rho_p\in A^{p+1}_{PGL_p}$, then there is a $p$-torsion class, say $\rho''_p$, in $A^{p+1}_{\Gamma_{p,n-p}}$ that restricts to $\rho'_p\in A^{p+1}_{\Gamma_p}$. On the other hand, recall that we have, besides the restriction $\opn{res}^G_H$, the transfer
\[\opn{tr}^H_G: A^*_H\rightarrow A^*_G\]
associated to a monomorphism $H\rightarrow G$ of algebraic groups of finite index.

Moreover, permuting the diagonal blocks of $\Gamma_{(p)}$ yields inner automorphisms of $\Gamma_{(p)}$. Let $s_k\in \Gamma_{(p)}$ be any permutation taking the $k$th diagonal block to the $1$st, by multiplication from the left.
\begin{proposition}\label{p,n-p to n}
Let $n$ and $p$ be as above. Let $u\in A^*_{\Gamma_{p,n-p}}$ be the image of some $p$-torsion class via the restriction from $A^*_{PGL_{p,n-p}}$. Then we have
\[\opn{res}^{\Gamma_n}_{\Gamma_{(p)}}\cdot\opn{tr}^{\Gamma_{p,n-p}}_{\Gamma_n}(u)=\sum_{k=1}^{n/p}\gamma_{s_k}
\cdot\opn{res}^{\Gamma_{p,n-p}}_{\Gamma_{(p)}}(u),\]
where $\gamma_{s_k}$ is the homomorphism induced by the conjugation action of $s_k$.
\end{proposition}
\begin{proof}
Throughout this proof we identify the double quotients
\[\Gamma_{(p)}\backslash\Gamma_n/\Gamma_{p,n-p}\cong S_{(p)}\backslash S_n/S_{p,n-p}.\]
Let $u\in A^*_{\Gamma_{p,n-p}}$ be the image of a $p$-torsion class via the restriction from $A^*_{PGL_{p,n-p}}$. We apply Mackey's formula (Proposition \ref{Mackey's formula}) to the class $u$ with $G=\Gamma_n$, $H=\Gamma_{p,n-p}$ and $K=\Gamma_{(p)}$ to obtain
\begin{equation}\label{applyMackey}
\opn{res}^{\Gamma_n}_{\Gamma_{(p)}}\cdot\opn{tr}^{\Gamma_{p,n-p}}_{\Gamma_n}(u)=
\sum_{K=(k_1,\cdots,k_{n/p})}\opn{tr}^{\Gamma_{p,n-p}^K}_{\Gamma_{(p)}}
\cdot\opn{res}^{s_K\Gamma_{p,n-p}s_K^{-1}}_{\Gamma_{p,n-p}^K}\cdot\gamma_{s_K}(u).
\end{equation}
In this formula $K$ runs over sequences of nonnegative integers $(k_1,\cdots,k_{n/p})$ summing up to $p$ (this uses Lemma \ref{double quotient}, (2)), and $\Gamma_{p,n-p}^K$ denotes
\[\Gamma_{p,n-p}^K:=s_K\Gamma_{p,n-p}s_K^{-1}\bigcap \Gamma_{(p)},\]
where $s_K$ is a representative of the double coset indexed by $K$.

The choice of $s_K$ is not essential. Nonetheless we make a normalized choice as follows. For the sequence $K=(k_1,\cdots,k_{n/p})$, consider the set
\[\abs{K}:=\{j|(i-1)\frac{n}{p}<j\leq (i-1)\frac{n}{p}+k_i,\ i=1,2,\cdots,\frac{n}{p}\}.\]
We assert that for $1\leq l \leq p$, the $l$th column of $s_K$ is the $l$th element of
\[\{e_j|j\in\abs{K}\}\]
with the ascending order in $j$. 
By Lemma \ref{index}, we have
\begin{equation}\label{(p) slash K}
\Gamma_{p,n-p}^K=\Gamma_{(p)/\abs{K}},
\end{equation}
where the notation $(p)/\abs{K}$ is defined in (\ref{W slash F}).

We proceed to consider the class
\begin{equation}\label{vanishing class}
\opn{tr}^{\Gamma_{p,n-p}^K}_{\Gamma_{(p)}}\cdot\opn{res}^{s_K\Gamma_{p,n-p}s_K^{-1}}_{\Gamma_{p,n-p}^K}\gamma_{s_K}(u)
\end{equation}
for each $K$. Suppose $K$ is not of the form $(0,\cdots,p,\cdots,0)$, then the sequence $(p)/\abs{K}$ contains some positive entry less then $p$. It then follows from Proposition \ref{BPGL_W} that the ring $A^*_{PGL_{(p)/\abs{K}}}$ contains no nontrivial $p$-torsion class. However, the restriction from $A^*_{PGL_{p,n-p}}$ to $A^*_{\Gamma_{(p)/\abs{K}}}$ factors through $A^*_{PGL_{(p)/\abs{K}}}$, and by (\ref{(p) slash K}) the class in (\ref{vanishing class}) vanishes. Therefore we have
\begin{equation}\label{applyMackey'}
\opn{res}^{\Gamma_n}_{\Gamma_{(p)}}\cdot\opn{tr}^{\Gamma_{p,n-p}}_{\Gamma_n}(u)=
\sum_{K=(0,\cdots p,\cdots 0)}\opn{tr}^{\Gamma_{(p)/\abs{K}}}_{\Gamma_{(p)}}\cdot
\opn{res}^{s_K\Gamma_{p,n-p}s_K^{-1}}_{\Gamma_{(p)/\abs{K}}}\gamma_{s_K}(u).
\end{equation}
When $K=(0,\cdots p,\cdots 0)$, one observes that $(p)/\abs{K}$ is $(p)$ with some $0$'s inserted in it. Therefore we have $\Gamma_{(p)/\abs{K}}=\Gamma_{(p)}$
and (\ref{applyMackey'}) further reduces to
\begin{equation*}
\opn{res}^{\Gamma_n}_{\Gamma_{(p)}}\cdot\opn{tr}^{\Gamma_{p,n-p}}_{\Gamma_n}(u)=
\sum_{K=(0,\cdots p,\cdots 0)}\opn{res}^{s_K\Gamma_{p,n-p}s_K^{-1}}_{\Gamma_{(p)/\abs{K}}}\gamma_{s_K}(u).
\end{equation*}
For $K=(0,\cdots p,\cdots 0)$ of which the $k$th entry is $p$, we may take $s_k=s_K$. Since the inclusion commutes with conjugations, we conclude.
\end{proof}

\begin{corollary}\label{existence of rho'}
If $p^2\nmid n$, then there is a $p$-torsion class $\rho'_{p,0}\in A^{p+1}_{\Gamma_n}$ that restricts to $\rho'_p\in A^{p+1}_{\Gamma_p}$ via the diagonal inclusion.
\end{corollary}
\begin{proof}
Recall that in \eqref{eq:def rho'} we define classes $\rho_p'\in A^{p+1}_{\Gamma_p}$ and $\rho_p''\in A^{p+1}_{\Gamma_{p,n-p}}$. By definition $\rho_p'$ is the restriction of some $p$-torsion class in $A^*_{PGL_p}$. Taking the restriction along (\ref{p,n-p blocks to p}) shows that $\rho_p''$ is the restriction of some $p$-torsion class in $A^*_{PGL_{p,n-p}}$. It follows from Proposition \ref{p,n-p to n} that we have
\[\opn{res}^{\Gamma_{(p)}}_{\Gamma_p}\cdot\opn{res}^{\Gamma_n}_{\Gamma_{(p)}}\cdot
\opn{tr}^{\Gamma_{p,n-p}}_{\Gamma_n}(\rho_p'')=
\sum_{k=1}^{n/p}\opn{res}^{\Gamma_{(p)}}_{\Gamma_p}\cdot\gamma_{s_k}\cdot
\opn{res}^{\Gamma_{p,n-p}}_{\Gamma_{(p)}}(\rho_p'').\]
Since permutation of the diagonal blocks acts trivially on the image of the diagonal inclusion from $\Gamma_p$ to $\Gamma_{(p)}$, we may ignore the $\gamma_{s_k}$'s and the above reduces to
\[\opn{res}^{\Gamma_n}_{\Gamma_p}\cdot\opn{tr}^{\Gamma_{p,n-p}}_{\Gamma_n}(\rho_p'')=
\sum_{k=1}^{n/p}\opn{res}^{\Gamma_{p,n-p}}_{\Gamma_p}(u)=\frac{n}{p}\opn{res}^{\Gamma_{p,n-p}}_{\Gamma_p}(\rho_p'')
=\frac{n}{p}\rho_p'.\]
Since we have $p^2\nmid n$ and the class $\rho_p''$ is $p$-torsion, we may take
\[\rho_{p,0}':=(\frac{n}{p})^{-1}\opn{tr}^{\Gamma_{p,n-p}}_{\Gamma_n}(\rho_p'')\]
and it is as desired.
\end{proof}
In particular, Corollary \ref{existence of rho'} indicates that the composition of inclusions
\[C_p\times\mu_p\hookrightarrow\Gamma_p\xrightarrow{\Delta}\Gamma_n\]
induces the restriction
\begin{equation}\label{rho'}
A^*_{\Gamma_n}\rightarrow A^*_{C_p\times\mu_p}, \quad \rho_{p,0}'\mapsto r=\xi\eta(\xi^{p-1}-\eta^{p-1}).
\end{equation}
We proceed to construct the promised classes $\rho_{p,k}\in A^{p+1}_{PGL_n}$ for $k\geq 0$. Recall that we have the diagonal map $\Delta: PGL_p\rightarrow PGL_n$ and the cycle class map $\opn{cl}: A^*_G\rightarrow H^{2*}_G$ for any algebraic group $G$ over $\mathbb{C}$.
\begin{lemma}\label{construct rho0}
When $p^2\nmid n$, there is a $p$-torsion class $\rho_{p,0}\in A^{p+1}_{PGL_n}$ satisfying $\Delta^*(\rho_{p,0})=\rho_p$ and $\opn{cl}(\rho_{p,0})=y_{p,0}$.
\end{lemma}
\begin{proof}
For $n=p$ there is nothing to show. We therefore assume $n>p$ for the rest of the proof. We will necessarily consider both $y_{p,0}\in H^*_{PGL_p}$ and $y_{p,0}\in H^*_{PGL_n}$. To avoid ambiguity we denote them by $y_{p,0}(p)$ and $y_{p,0}(n)$, respectively.

Consider the following commutative diagram
\begin{equation}\label{diag1}
\begin{tikzcd}
A^*_{PGL_n}\arrow[d, two heads, "f^*"]\arrow[r]& A^*_{\Gamma_n}\arrow[r]\arrow[d, two heads, "g^*"]& A^*_{C_p\times\mu_p}\arrow[d, two heads]\\
A^*_{PGL_n}(sl_n^*)\arrow[r, "\varphi"]&A^*_{\Gamma_n}(D_n^*)\arrow[r]& A^*_{C_p\times\mu_p}(D_n^*)
\end{tikzcd}
\end{equation}
in which all arrows are the obvious restrictions. The vertical arrows are surjective, due to the localization sequences, and $\varphi$ is an isomorphism, due to Proposition \ref{Vezzosi}. Therefore, we have a class $\wt{\rho}_{p,0}\in A^{p+1}_{PGL_n}$ satisfying
\begin{equation}\label{rhop1}
\varphi f^*(\wt{\rho}_{p,0})=g^*(\rho'_{p,0}).
\end{equation}

It follows from Corollary \ref{restriction to p} that the vertical arrow of \eqref{diag1} on the right hand side is an isomorphism in degree $p+1$, and a diagram chasing therefore shows that $\wt{\rho}_{p,0}$ restricts to $r=\xi\eta(\xi^{p-1}-\eta^{p-1})$. On the other hand, in the following commutative diagram
\begin{equation*}
\begin{tikzcd}
A^*_{PGL_n}\arrow[dr]\arrow[rr,"\Delta^*"]& & A^*_{PGL_p}\arrow[dl]\\
& A^*_{C_p\times\mu_p}
\end{tikzcd}
\end{equation*}
there is a unique torsion class $\rho_p\in A^*_{PGL_p}$ sent to $r$, a consequence of Theorem \ref{TotaroInjection}, Proposition \ref{VistoliInjection} and Proposition \ref{Vistoli(p-1)!}. It then follows that we have
\begin{equation}\label{eq:wtrho 1}
\Delta^*(\wt{\rho}_{p,0})=\rho_p.
\end{equation}


In fact, we would have been able to define the class $\rho_{p,0}$ to be $\wt{\rho}_{p,0}$, if $\wt{\rho}_{p,0}$ were known to be a torsion class. Roughly speaking, the torsion class $\rho_{p,0}$ is to be constructed by ``annihilating the non-torsion part of $\wt{\rho}_{p,0}$''. In what follows we add the subscripts ``$A$'' and ``$H$'' to the conventional notations to indicate whether they are meant for Chow rings or singular cohomology.

Recall the notations $A^*_G[p]$ and $H^*_G[p]$ for localization at $p$. Let $\lambda: T(PGL_n)\rightarrow PGL_n$ be the inclusion of the chosen maximal torus. Then we have the induced restrictions
\[\lambda^*_A: A^*_{PGL_n}[p]\rightarrow (A^*_{T(PGL_n)}[p])^{S_n}\]
and
\[\lambda^*_H: H^*_{PGL_n}[p]\rightarrow (H^*_{T(PGL_n)}[p])^{S_n}.\]
It follows from Corollary \ref{cor:H^2p+2 direct sum 2} that the latter is surjective in degree $2(p+1)$. As for the former, we denote the image of $\lambda^*_A$ by $A^*_I[p]$, a subring of $(A^*_{T(PGL_n)}[p])^{S_n}$. Consider the diagram
\begin{equation}\label{dia:classmapdiagram}
\begin{tikzcd}
A^{p+1}_{C_p\times\mu_p}\arrow[d,"\opn{cl}_1","\cong"']&A^{p+1}_{PGL_n}[p]\arrow[l,"\theta^*_A"]\arrow[r,two heads,"\lambda^*_A"]\arrow[d,"\opn{cl}_0"]& A^{p+1}_I[p]\arrow[d,hook',"\opn{cl}_2"]\arrow[l,bend right,dashed,"\phi_A"']
\\
H^{2(p+1)}_{C_p\times\mu_p}&H^{2(p+1)}_{PGL_n}[p]\arrow[r,two heads,"\lambda^*_H"]\arrow[l,"\theta^*_H"]& (H^{2(p+1)}_{T(PGL_p)}[p])^{S_n}\arrow[l,bend left,"\phi_H"]
\end{tikzcd}
\end{equation}
which is commutative apart from the bent arrows. where the vertical arrows are the cycle class maps and the horizontal ones are the restrictions. The arrow $\lambda^*_A$ is surjective by construction, and $\lambda_H^*$ is a split epimorphism with a right inverse $\phi_H$ satisfying $\theta^*_H\phi_H=0$, as shown in Corollary \ref{cor:H^2p+2 direct sum 2}. The map $\opn{cl}_2$ is injective by Corollary \ref{cor:cl T(PGLn)}. The map $\opn{cl}_1$ is an isomorphism, by Proposition \ref{A(Cpmup)} and Proposition \ref{H(Cpmup)}.

Since $A^{p+1}_I[p]$ is a free $\Z_{(p)}$-module, we have a homomorphism $\phi_A$ of $\Z_{(p)}$-modules as indicated by the dashed arrow in \eqref{dia:classmapdiagram}, which is a lift of $\phi_H\opn{cl}_2$ along $\opn{cl}_0$. In other words, we have $\opn{cl}_0\phi_A=\phi_H\opn{cl}_2$. It follows that $\lambda_A^*$ is a split epimorphism with right inverse $\phi_A$. To show this, notice that we have
\[\opn{cl}_2\lambda^*_A\phi_A=\lambda^*_H\opn{cl}_0\phi_A=\lambda^*_H\phi_H\opn{cl}_2=\opn{cl}_2,\]
which yields $\lambda^*_A\phi_A=\opn{id}$, since $\opn{cl}_2$ is injective.

Define $\rho_{p,0}:=\wt{\rho}_{p,0}-\phi_A\lambda^*_A(\wt{\rho}_{p,0})$. Therefore
\[\lambda^*_A(\rho_{p,0})=\lambda^*_A(\wt{\rho}_{p,0}-\phi_A\lambda^*_A(\wt{\rho}_{p,0}))
=\lambda^*_A(\wt{\rho}_{p,0})-\lambda^*_A(\wt{\rho}_{p,0})=0.\]
Since
\[\lambda^*_A: A^*_{PGL_n}\otimes\mathbb{Q}\rightarrow (A^*_{T(PGL_n)})^{S_n}\otimes\mathbb{Q}\]
is a ring isomorphism (Corollary \ref{cor:T(PGL_n)Q}), it follows that $\rho_{p,0}$ is a torsion class. Furthermore, by Proposition \ref{Vezzosi n torsion}, the class $\rho_{p,0}$ is $n$-torsion. Since we have $p^2\nmid n$ and we are considering the $p$-local case, the class $\rho_{p,0}$ is $p$-torsion.

Since we have $\theta^*_H\phi_H=0$ and $\opn{cl}_1$ is an isomorphism, we obtain $\theta^*_A\phi_A=0$. Now it follows from \eqref{eq:wtrho 1} that we have
\begin{equation*}
\begin{split}
&\opn{res}^{PGL_p}_{C_p\times\mu_p}\Delta^*(\rho_{p,0})=\theta^*_A(\rho_{p,0})\\
=&\theta^*_A(\wt{\rho}_{p,0}-\phi_A\lambda^*_A(\wt{\rho}_{p,0}))
=\theta^*_A(\wt{\rho}_{p,0})=r=\opn{res}^{PGL_p}_{C_p\times\mu_p}(\rho_p).
\end{split}
\end{equation*}
Since $\Delta^*(\rho_{p,0})$ is a torsion class, and by Proposition \ref{pro:Vistoli inj}, $\rho_p$ is the only torsion class in $A^*_{PGL_p}$ restricting to $r$, we have $\Delta^*(\rho_{p,0})=\rho_p$.

On the other hand, we have
\[\theta^*_H\opn{cl}_0(\rho_{p,0})=\opn{cl}_1\theta^*_A(\wt{\rho}_{p,0}-\phi_A\lambda^*_A(\wt{\rho}_{p,0}))
=\opn{cl}_1\theta^*_A(\wt{\rho}_{p,0})=r.\]
Therefore we have $\opn{cl}_0(\rho_{p,0})\in(\theta^*_H)^{-1}(r)$. However, by Corollary \ref{cor:H^2p+2 direct sum 2}, we have the isomorphism
\begin{equation*}
\begin{split}
H^{2(p+1)}_{PGL_n}[p]&\cong (H^{2(p+1)}_{T(PGL_n)}[p])^{S_n}\oplus (H^{2(p+1)}_{C_p\times\mu_p})^{SL_2(\Z/p)}\\
&=(H^{2(p+1)}_{T(PGL_n)}[p])^{S_n}\oplus (r),
\end{split}
\end{equation*}
which interprets the restriction $\theta^{2(p+1)}:H^{2(p+1)}_{PGL_n}[p]\rightarrow H^{2(p+1)}_{C_p\times\mu_p}$ as the projection onto the second summand. Therefore there is only one torsion class in $(\theta^*_H)^{-1}(r)$, which is $y_{p,0}(n)$. Hence we have
$\opn{cl}(\rho_{p,0})=\opn{cl}_0(\rho_{p,0})=y_{p,0}(n)$.


\end{proof}

What remains to complete the proof of Theorem \ref{thm2p^k+2} is the construction of the $p$-torsion classes $\rho_{p,k}$ for $k>0$, which is presented in the following
\begin{lemma}\label{rhopk}
For $p^2\nmid n$ and $k\geq 0$, there are $p$-torsion classes $\rho_{p,k}\in A^{p^{k+1}+1}_{PGL_n}$ satisfying $\opn{cl}(\rho_{p,k})=y_{p,k}$.
\end{lemma}
\begin{proof}
Let us recall that for the short exact sequence $0\rightarrow\Z_{(p)}\xrightarrow{\times p}\Z_{(p)} \rightarrow\Z/p\rightarrow0$, there is an associated long exact sequence of motivic cohomology groups as given in \eqref{eq:mot les}
\begin{equation*}
\cdots\xrightarrow{B}H^{s,t}_{PGL_n}[p]\xrightarrow{\times p}H^{s,t}_{PGL_n}[p]\rightarrow H^{s,t}_{PGL_n}(p)\xrightarrow
{B}H^{s+1,t}_{PGL_n}[p]\xrightarrow{\times p}\cdots.
\end{equation*}
Hence, a class $a\in H^{s+1,t}_{PGL_n}[p]$ is a $p$-torsion class  if and only if $a=B(b)$ for some $b\in H^{s,t}_{PGL_n}(p)$.

Therefore, we have $b_{p,0}\in H^{2p+1,p+1}_{PGL_n}(p)$ satisfying $B(b_{p,0})=\rho_{p,0}$. By Proposition \ref{Steenrodcycleclass}, we have
\begin{equation}\label{eq:b(p,0)}
B\cdot\opn{cl}(b_{p,0})=\opn{cl}\cdot B(b_{p,0})=\opn{cl}(\rho_{p,0})=y_{p,0}.
\end{equation}
Now we consider the group $H^{2p+1}_{PGL_n}[p]$. Theorem 1.2 of \cite{Gu} asserts that the torsion subgroups of $H^k_{PGL_n}[p]$ are $0$ for $3<k<2p+2$. On the other hand, the rational cohomology ring $H^*_{PGL_n}\otimes\mathbb{Q}$ concentrates in even degrees. Therefore we have
\begin{equation}\label{eq:2p+1}
H^{2p+1}_{PGL_n}[p]=0,
\end{equation}
from which we deduce that the Bockstein homomorphism
\begin{equation}\label{eq:delta 2p+1}
B:H^{2p+1}_{PGL_n}(p)\hookrightarrow H^{2(p+1)}_{PGL_n}[p]
\end{equation}
is injective. The class $\opn{cl}(b_{p,0})$ is determined, by \eqref{eq:b(p,0)} and \eqref{eq:delta 2p+1}, as the unique class in $B^{-1}(y_{p,0})$. This class has been mentioned before. Recall that in Section \ref{SecKZ3} we define classes $x_{p,k}\in H^{2p^{k+1}+1}(K(\Z,3);\Z/p)$ for $k\geq 0$, satisfying $B(x_{p,k})=y_{p,k}$. As before, we denote by $x_{p,k}$ the image of itself via the map
\[\chi^*:H^{2p^{k+1}+1}(K(\Z,3);\Z/p)\rightarrow H^{2p^{k+1}+1}_{PGL_n}(p),\]
omitting the notation $\chi^*$. Therefore, we have
\begin{equation}\label{eq:cl b(p,0)}
\opn{cl}(b_{p,0})=x_{p,0}=\scP^1(\bar{x}_1).
\end{equation}

Inductively, we define \[b_{p,k}=\scS^{p^k}(b_{p,k-1})\in H^{2p^{k+1}+1,p^{k+1}+1}_{PGL_n}\]
and
\[\rho_{p,k}=B(b_{p,k})\in H^{2p^{k+1}+2,p^{k+1}+1}_{PGL_n}=A^{p^{k+1}+1}_{PGL_n}.\]

We verify $\opn{cl}(b_{p,k})=x_{p,k}$. For $k=0$, this follows from \eqref{eq:cl b(p,0)}. By induction on $k$, we have
\[\opn{cl}(b_{p,k})=\opn{cl}\cdot\scS^{p^k}(b_{p,k-1})=\scP^{p^k}\cdot\opn{cl}(b_{p,k-1})=\scP^{p^k}(x_{p,k-1})=
x_{p,k}.\]
Therefore we have
\[\opn{cl}(\rho_{p,k})=\opn{cl}\cdot B(b_{p,k})=B\cdot\opn{cl}(b_{p,k})=B(x_{p,k})=y_{p,k}.\]
\end{proof}
\begin{remark}
The proof above together with the Adem relation \eqref{eq:MotAdem1} show that we have
\[\scS^{p^k}(\bar{\rho}_{p,k-1})=\bar{\rho}_{p,k}.\]
\end{remark}




We proceed to give a corollary of Theorem \ref{thm2p^k+2}.
\begin{corollary}\label{transdimension}
For singular cohomology, the composite of the restrictions
\[\theta_H^*: H^*_{PGL_n}\xrightarrow{\Delta^{*}} H^*_{PGL_p}\rightarrow H^*_{C_p\times\mu_p} \]
takes $y_{p,k}$ to $\xi\eta(\xi^{p^{k+1}-1}-\eta^{p^{k+1}-1})$, and $x_1$ (the canonical generator of $H^3_{PGL_p})$ to $\zeta$.

For the Chow ring, similarly, when $p^2\nmid n$, the composite
\[\theta_A^*: A^*_{PGL_n}\xrightarrow{\Delta^{*}} A^*_{PGL_p}\rightarrow A^*_{C_p\times\mu_p} \]
takes $\rho_{p,k}$ to $\xi\eta(\xi^{p^{k+1}-1}-\eta^{p^{k+1}-1})$.

Finally, when $n=p$, the homomorphism $H^*_{PGL_p}\rightarrow H^*_{C_p\times\mu_p}$ (resp. $A^*_{PGL_p}\rightarrow A^*_{C_p\times\mu_p}$) restricts to a monomorphism on the subalgebra generated by $\{y_{p,k}\}_{k\geq0}$ (resp. $\{\rho_{p,k}\}_{k\geq0}$).
\end{corollary}
\begin{proof}
In the case $n=p$, it follows immediately from Vistoli's work, as given in \eqref{eq:Vistoli Pro11.1} that the class $\rho_{p,0}$ restricts to $\xi\eta(\xi^{p-1}-\eta^{p-1})$. By the construction of $\rho_{p,k}$ via the Steenrod power operations, the classes $\rho_{p,k}$ restrict to $\xi\eta(\xi^{p^{k+1}-1}-\eta^{p^{k+1}-1})$.

Applying the cycle class map, one sees that the classes $y_{p,k}$ restrict to $\xi\eta(\xi^{p^{k+1}-1}-\eta^{p^{k+1}-1})$. Since we have
\[y_{p,0}=B\mathscr{P}^{1}(\bar{x}_1),\hspace{5 mm}\xi\eta(\xi^{p-1}-\eta^{p-1})=B\mathscr{P}^{1}(\bar{\zeta}),\]
the class $x_1$ restricts to $\zeta$.

The general case follows since the diagonal map $\Delta:PGL_p\rightarrow PGL_n$ restricts
$y_{p,k}\in H^*_{PGL_n}$ to $y_{p,k}\in H^*_{PGL_p}$ and $x_1\in H^*_{PGL_n}$ to $x_1\in H^*_{PGL_p}$, and when $p^2\nmid n$, the analogous statement can be made for $\rho_{p,k}$.

The last paragraph is immediately deduced from Proposition \ref{pro:Vistoli inj}.
\end{proof}

\section{Classes not in the Chern subring of $A^*_{PGL_n}$}\label{SecChern}
Let $G$ be a complex algebraic group. Recall that in the introduction we mentioned the refined cycle class map
\[\tilde{\opn{cl}}: A^*_G\rightarrow MU^*(\mathbf{B}G)\otimes_{MU^*}\mathbb{Z}.\]
In \cite{Tot} (page 2), Totaro conjectured that for $G$ such that the complex cobordism ring of $\mathbf{B}G$ is concentrated in even degrees after tensoring with $\mathbb{Z}_{(p)}$, for a fixed odd prime $p$, the map $\tilde{\opn{cl}}$ is an isomorphism after tensoring with $\mathbb{Z}_{(p)}$.

On the other hand, in \cite{Ko1}, Kono and Yagita showed that for $p=3$, the Brown-Peterson cohomology ring $BP^*(\mathbf{B}PGL_3)$ is not generated by Chern classes. This implies that, if Totaro's conjecture holds for $G=PGL_3$ and $p=3$, then the localized Chow ring $A^*_{PGL_3}\otimes\mathbb{Z}_{(3)}$ is not generated by Chern classes. This is verified by Vezzosi in \cite{Ve}. More precisely, he showed that the class $\rho_3$ is not in the Chern subring. In \cite{Ta}, Targa showed that the same holds for any odd prime $p$. In \cite{Ka1}, Kameko and Yagita showed a stronger result for $H^*(\mathbf{B}PU_p;\mathbb{Z})$, which easily generalizes to $A^*_{PGL_p}$ and $H^*_{PGL_p}$.

As a generalization of these results we have the following
\begin{theorem}[Theorem \ref{thmChern}]
Let $n>1$ be an integer, and $p$ one of its odd prime divisor, such that $p^2\nmid n$. Then the ring $A^*_{PGL_n}\otimes\mathbb{Z}_{(p)}$ is not generated by Chern classes. More precisely, the class $\rho_{p,0}^i$ is not in the Chern subring for $p-1\nmid i$.
\end{theorem}
\begin{proof}
Recall that $\bar{y}_{p,0}\in H^{2(p+1)}(\mathbf{B}PU_p;\mathbb{Z}/p)$ denotes the mod $p$ reduction of $\bar{y}_{p,0}$.
As mentioned above, in \cite{Ka1}, Kameko and Yagita show that $\bar{y}_{p,0}^i$, for $p-1\nmid i$, is not in the Chern subring of $H^{2(p+1)}(\mathbf{B}PU_p;\mathbb{Z}/p)$. Since $y_{p,0}^i\in H^*_{PGL_p}$ is a $p$-torsion class itself, it is not in the Chern subring of $H^*_{PGL_p}$. Applying the cycle class map, we see that similarly $\rho_p^i$ is not in the Chern subring of $A^*_{PGL_p}$.

On the other hand, it follows from Lemma \ref{construct rho0} that we have $\Delta^*(\rho_{p,0}^i)=\rho_p^i$, where \[\Delta: \mathbf{B}PGL_p\rightarrow \mathbf{B}PGL_n\]
is induced by the obvious diagonal map. It then follows that $\rho_{p,0}^i$ is not in the Chern subring.
\end{proof}
\begin{remark}
In \cite{Ka1}, Kameko and Yagita considered the class  $Q_0Q_1x_2$ where $x_2$ is a generator of $H^2(\mathbf{B}PU_p;\Z/p)$, and $Q_0,Q_1$ are among the Milnor basis constructed in \cite{milnor1958steenrod}. It follows from the argument in Remark \ref{rem:Milnor basis} that we have $\bar{y}_{p,0}=Q_0Q_1x_2$.
\end{remark}
\bibliographystyle{abbrv}
\bibliography{JLMSsom_tor_ref_v4}

\begin{thebibliography}{10}

\bibitem{Ad}
A.~Adem and R.~J. Milgram.
\newblock {\em Cohomology of finite groups}, volume 309.
\newblock Springer Science \& Business Media, 2013.

\bibitem{Ad1}
J.~Adem.
\newblock The iteration of the {Steenrod} squares in algebraic topology.
\newblock {\em Proceedings of the National Academy of Sciences of the United
  States of America}, 38 8:720--6, 1952.

\bibitem{antieau2015integral}
B.~Antieau.
\newblock On the integral {T}ate conjecture for finite fields and
  representation theory.
\newblock {\em Algebraic Geometry}, 3(2):138--149, 2016.

\bibitem{An}
B.~Antieau and B.~Williams.
\newblock The topological period--index problem over $6$-complexes.
\newblock {\em Journal of Topology}, 7(3):617--640, 2013.

\bibitem{An1}
B.~Antieau and B.~Williams.
\newblock The period-index problem for twisted topological {K}--theory.
\newblock {\em Geometry \& Topology}, 18(2):1115--1148, 2014.

\bibitem{Br}
P.~Brosnan.
\newblock {Steenrod} operations in {Chow} theory.
\newblock {\em Transactions of the American Mathematical Society},
  355(5):1869--1903, 2003.

\bibitem{Co}
J.-L. Colliot-Th{\'e}lene.
\newblock Exposant et indice d'alg{\`e}bres simples centrales non
  ramifi{\'e}es.
\newblock {\em Enseignement Mathematique}, 48(1/2):127--146, 2002.

\bibitem{cordova2020anomalies}
C.~Cordova, D.~Freed, H.~T. Lam, and N.~Seiberg.
\newblock Anomalies in the space of coupling constants and their dynamical
  applications ii.
\newblock {\em SciPost Physics Proceedings}, 8(1), 2020.

\bibitem{Cr}
D.~Crowley and M.~Grant.
\newblock The topological period-index conjecture for $\operatorname{Spin}^c$
  $6$-manifolds.
\newblock {\em to appear in the Annals of {K}-Theory, arXiv:1802.01296}, 2018.

\bibitem{davighi2019global}
J.~Davighi, B.~Gripaios, and N.~Lohitsiri.
\newblock Global anomalies in the standard model (s) and beyond.
\newblock {\em arXiv preprint arXiv:1910.11277}, 2019.

\bibitem{Ed}
D.~Edidin and W.~Graham.
\newblock Equivariant intersection theory (with an appendix by {Angelo
  Vistoli}: The {Chow} ring of $\mathcal{M}_2$).
\newblock {\em Inventiones mathematicae}, 131(3):595--634, 1998.

\bibitem{Fi}
R.~Field.
\newblock The {Chow} ring of the classifying space {$BSO(2n,\mathbb{C})$}.
\newblock {\em Journal of Algebra}, 350(1):330 -- 339, 2012.

\bibitem{Ca}
{\'E}.~N.~S. (France) and H.~Cartan.
\newblock {\em S{\'e}minaire {Henri Cartan}: ann. 7 1954/1955; Alg{\`e}bres
  {d'Eilenberg-Maclane} et homotopie}.
\newblock Secretariat Mathematique, 1958.

\bibitem{Fu}
W.~Fulton.
\newblock {\em Intersection theory}, volume~2.
\newblock Springer Science \& Business Media, 2013.

\bibitem{Ga}
I.~Garc{\'\i}a-Etxebarria and M.~Montero.
\newblock {D}ai-{F}reed anomalies in particle physics.
\newblock {\em Journal of High Energy Physics}, 2019(8):3, 2019.

\bibitem{Go}
D.~H. Gottlieb.
\newblock Fibre bundles and the {Euler} characteristic.
\newblock {\em J. Differential Geom.}, 10(1):39--48, 1975.

\bibitem{Gu2}
X.~Gu.
\newblock The topological period--index problem over 8-complexes, {II}.
\newblock {\em to appear in the Proceedings of the American Mathematical
  Society, arXiv:1803.05100}.

\bibitem{Gu}
X.~Gu.
\newblock On the cohomology of the classifying spaces of projective unitary
  groups.
\newblock {\em Journal of Topology and Analysis}, pages 1--39, 2019.

\bibitem{Gu1}
X.~Gu.
\newblock The topological period--index problem over 8-complexes, {I}.
\newblock {\em Journal of Topology}, 12(4):1368--1395, 2019.

\bibitem{Gui}
P.~Guillot.
\newblock {Chow} rings and cobordism of some {Chevalley} groups.
\newblock In {\em Mathematical Proceedings of the Cambridge Philosophical
  Society}, volume 136, pages 625--642. Cambridge University Press, 2004.

\bibitem{Gui1}
P.~Guillot.
\newblock {Steenrod} operations on the {Chow} ring of a classifying space.
\newblock {\em Advances in Mathematics}, 196(2):276--309, 2005.

\bibitem{Gui2}
P.~Guillot.
\newblock The {C}how rings of ${G}_2$ and ${S}pin(7)$.
\newblock {\em Journal f{\"u}r die reine und angewandte Mathematik (Crelles
  Journal)}, 2007(604):137--158, 2007.

\bibitem{hsiang2012cohomology}
W.~Y. Hsiang.
\newblock {\em Cohomology theory of topological transformation groups},
  volume~85.
\newblock Springer Science \& Business Media, 2012.

\bibitem{Ka2}
M.~Kameko.
\newblock Representation theory and the cycle map of a classifying space.
\newblock {\em Algebraic Geometry}, 4(2):221--228, 2017.

\bibitem{Ka}
M.~Kameko and N.~Yagita.
\newblock The {Brown-Peterson} cohomology of the classifying spaces of the
  projective unitary groups {$PU(p)$} and exceptional {Lie} groups.
\newblock {\em Transactions of the American Mathematical Society},
  360(5):2265--2284, 2008.

\bibitem{Ka1}
M.~Kameko and N.~Yagita.
\newblock {Chern} subrings.
\newblock {\em Proceedings of the American Mathematical Society},
  138(1):367--373, 2010.

\bibitem{Ko}
A.~Kono and M.~Mimura.
\newblock On the cohomology of the classifying spaces of {$PSU (4n+ 2)$} and
  {$PO (4n+ 2)$}.
\newblock {\em Publications of the Research Institute for Mathematical
  Sciences}, 10(3):691--720, 1975.

\bibitem{Ko1}
A.~Kono and N.~Yagita.
\newblock {Brown-Peterson} and ordinary cohomology theories of classifying
  spaces for compact {Lie} groups.
\newblock {\em Transactions of the American Mathematical Society},
  339(2):781--798, 1993.

\bibitem{mazza2011lecture}
C.~Mazza, V.~Voevodsky, and C.~A. Weibel.
\newblock {\em Lecture notes on motivic cohomology}, volume~2.
\newblock American Mathematical Soc., 2011.

\bibitem{Mc}
J.~McCleary.
\newblock {\em A user's guide to spectral sequences}.
\newblock Number~58. Cambridge University Press, 2001.

\bibitem{milnor1958steenrod}
J.~Milnor.
\newblock The {S}teenrod algebra and its dual.
\newblock {\em Annals of Mathematics}, pages 150--171, 1958.

\bibitem{Ro1}
L.~A. Molina~Rojas and A.~Vistoli.
\newblock On the {Chow} rings of classifying spaces for classical groups.
\newblock {\em Rendiconti del Seminario Matematico della Universit{\`a} di
  Padova}, 116:271--298, 2006.

\bibitem{Pa}
R.~Pandharipande.
\newblock Equivariant {Chow} rings of {$O(k)$,$SO(2k+1)$, and $SO(4)$}.
\newblock {\em J. reine angew. Math}, 131:148, 1998.

\bibitem{Ro}
L.~A.~M. Rojas, A.~Vistoli, and R.~Spigler.
\newblock The {Chow} ring of the classifying space of
  {$\operatorname{Spin}_8$}, 2006.

\bibitem{Sta}
T.~{Stacks project authors}.
\newblock The stacks project.
\newblock \url{https://stacks.math.columbia.edu}, 2019.

\bibitem{St}
N.~E. Steenrod and D.~B. Epstein.
\newblock {\em Cohomology operations}.
\newblock Princeton University Press, 1962.

\bibitem{Sw}
R.~M. Switzer.
\newblock {\em Algebraic topology-homotopy and homology}.
\newblock Springer, 2017.

\bibitem{tamanoi1999subalgebras}
H.~Tamanoi.
\newblock {$Q$}-subalgebras, {Milnor} basis, and cohomology of
  {Eilenberg-MacLane} spaces.
\newblock {\em Journal of Pure and Applied Algebra}, 137(2):153--198, 1999.

\bibitem{Ta}
E.~Targa.
\newblock {Chern} classes are not enough.
\newblock {\em Journal f{\"u}r die reine und angewandte Mathematik (Crelles
  Journal)}, 2007(610):229--233, 2007.

\bibitem{To}
H.~Toda.
\newblock Cohomology of classifying spaces.
\newblock In {\em Homotopy Theory and Related Topics}, pages 75--108, Tokyo,
  Japan, 1987. Mathematical Society of Japan.

\bibitem{Tot}
B.~Totaro.
\newblock The {Chow} ring of a classifying space.
\newblock In {\em Proceedings of Symposia in Pure Mathematics}, volume~67,
  pages 249--284. Providence, RI; American Mathematical Society; 1998, 1999.

\bibitem{totaro2014group}
B.~Totaro.
\newblock {\em Group cohomology and algebraic cycles}, volume 204.
\newblock Cambridge University Press, 2014.

\bibitem{Va}
A.~Vavpeti{\v{c}} and A.~Viruel.
\newblock On the mod $p$ cohomology of {$BPU(p)$}.
\newblock {\em Transactions of the American Mathematical Society}, pages
  4517--4532, 2005.

\bibitem{Ve}
G.~Vezzosi.
\newblock On the {Chow} ring of the classifying stack of
  {$PGL_{3,\mathbb{C}}$}.
\newblock {\em Journal f{\"u}r die reine und angewandte Mathematik (Crelles
  Journal)}, 2000(523):1--54, 2000.

\bibitem{Vi}
A.~Vistoli.
\newblock On the cohomology and the {Chow} ring of the classifying space of
  {$PGL_p$}.
\newblock {\em Journal f{\"u}r die reine und angewandte Mathematik (Crelles
  Journal)}, 2007(610):181--227, 2007.

\bibitem{voevodsky1999voevodosky}
V.~Voevodsky.
\newblock Voevodosky¡¯s seattle lectures: K-theory and motivic cohomology proc.
  of symposiain pure math.¡±.
\newblock {\em Algebraic K-theory¡±(1997: University of Washington, Seattle)},
  67:283--303, 1999.

\bibitem{Vo}
V.~Voevodsky.
\newblock Reduced power operations in motivic cohomology.
\newblock {\em Publications Math{\'e}matiques de l'IH{\'E}S}, 98:1--57, 2003.

\bibitem{Ya}
N.~Yagita.
\newblock Applications of {Atiyah--Hirzebruch} spectral sequences for motivic
  cobordism.
\newblock {\em Proceedings of the London Mathematical Society}, 90(3):783--816,
  2005.

\end{thebibliography}
\end{document}